\documentclass[a4paper,11pt, twoside]{amsart}

\usepackage[inner=2.5cm,outer=2.5cm,top=3cm,headsep=1cm, footskip=1cm,bottom=3cm]{geometry}
\usepackage[english]{babel}
\usepackage{amsfonts,latexsym,rawfonts,amsmath,amssymb,amsthm,amscd}
\usepackage{fancyhdr}
\usepackage{enumerate}
\usepackage{stackrel}

\usepackage{hyperref}
\usepackage{mdwlist}
\usepackage{array} 
\input xy
\xyoption{all}

\hypersetup{
    colorlinks,%
    citecolor=black,%
    filecolor=black,%
    linkcolor=black,%
    urlcolor=black
}

\makeatletter
\let\uppercasenonmath\@gobble
\makeatother

\usepackage{titlesec}
\titleformat{\section}{\bfseries\center}{\thesection.}{0.4em}{}
\titleformat{\subsection}{\vspace{.08cm}\bfseries}{\thesubsection.}{0.4em}{}

\newtheorem{prop}{Proposition}[section]
\newtheorem{teo}[prop]{Theorem}
\newtheorem{lem}[prop]{Lemma}
\newtheorem{cor}[prop]{Corollary}

\theoremstyle{definition}
\newtheorem{nada}[prop]{}
\newtheorem{defi}[prop]{Definition}
\newtheorem{example}[prop]{Example}
\newtheorem{rmk}[prop]{Remark}
\theoremstyle{theorem}

\def\Ho{\mathrm{Ho}}
\def\Hom{\mathrm{Hom}}

\def\Dec{\mathrm{Dec}}
\def\Cyl{\mathrm{Cyl}}

\def\Img{\mathrm{Im }}
\def\Ext{\mathrm{Ext }}
\def\Inj{\mathrm{Inj }}

\newcommand{\eps}{\varepsilon}
\newcommand{\simr}[1]{\begin{array}{c}\vspace{-.3cm} \simeq\\ \vspace{-.4cm}\text{\tiny{#1}}\vspace{.36cm} \end{array}}
\newcommand{\hto}{\rightsquigarrow}

\newcommand{\lra}{\longrightarrow}

\newcommand{\Cx}[1]{\mathbf{C}^+(#1)}
\newcommand{\FCx}[1]{\mathbf{C}^+(\mathbf{F}#1)}
\newcommand{\FFCx}[1]{\mathbf{C}^+(\mathbf{F}^2#1)}
\newcommand{\FCxinj}[2]{\mathbf{C}^+_{#1}(\mathbf{F}\text{\normalfont Inj}#2)}
\newcommand{\FCxr}[2]{\mathbf{C}^+_{#1}(\mathbf{F}#2)}
\newcommand{\DFCx}{\Gamma\mathbf{F}}
\newcommand{\MHC}{\mathbf{MHC}}
\newcommand{\AHC}{\mathbf{AHC}}
\newcommand{\MHS}{\mathsf{MHS}}

\newcommand{\GCc}{\Gamma\Cc}

\newcommand{\CC}{\mathbb{C}}

\newcommand{\QQ}{\mathbb{Q}}
\newcommand{\RR}{\mathbb{R}}

\newcommand{\ZZ}{\mathbb{Z}}

\newcommand{\Aa}{\mathcal{A}}

\newcommand{\Cc}{\mathcal{C}}
\newcommand{\Dd}{\mathcal{D}}
\newcommand{\Ee}{\mathcal{E}}

\newcommand{\Hh}{\mathcal{H}}

\newcommand{\Jj}{\mathcal{J}}

\newcommand{\Mm}{\mathcal{M}}

\newcommand{\Qq}{\mathcal{Q}}
\newcommand{\Ss}{\mathcal{S}}
\newcommand{\Ww}{\mathcal{W}}

\newcommand{\kk}{\mathbf{k}}

\title[\sc \normalsize Homotopy Theory of Mixed Hodge Complexes]{\Large{Homotopy Theory of Mixed Hodge Complexes}}
\author[\sc \normalsize Joana Cirici] {\large Joana Cirici}
\address[J. Cirici]{
Fachbereich Mathematik und Informatik\\
Freie Universit\"{a}t Berlin\\  Arnimallee 3\\ 
14195 Berlin}
\email{jcirici@math.fu-berlin.de}

\author[Francisco Guill\'{e}n] {Francisco Guill\'{e}n}
\address[F. Guill\'{e}n]{Departament
d'\`{A}lgebra i Geometria\\  Universitat de Barcelona\\ Gran Via 585,
08007 Barcelona}
\email{fguillen@ub.edu}
\thanks{Partially supported by the Spanish Ministry of Economy and Competitiveness under project MTM 2009-09557 and by the Generalitat de
Catalunya as members of the team 2009 SGR 119. The first-named author 
wants to acknowledge financial support from the German Research Foundation through the project SFB 647.}

\subjclass[2010]{32S35, 55U35.}
\keywords{Mixed Hodge theory, homotopical algebra, mixed Hodge complex, filtered derived category,
weight filtration, absolute filtration, diagram category, Cartan-Eilenberg category}

\date{\today}

\begin{document}
\maketitle

\begin{abstract}
We show that the category of mixed Hodge complexes admits a Cartan-Eilenberg structure, a notion introduced in \cite{GNPR}
leading to
a good calculation of the homotopy category in terms of (co)fibrant objects. This result provides a conceptual 
framework from which Beilinson's \cite{Be} and Carlson's \cite{Ca} results on
mixed Hodge complexes and extensions of mixed Hodge structures follow easily.
\end{abstract}
\section{Introduction}
Mixed Hodge complexes were introduced by Deligne \cite{DeHIII} in order to extend his theory of mixed Hodge structures on the cohomology
of algebraic varieties to the singular case, via simplicial resolutions. 
Since their appearance, these objects and their variants (see for example \cite{Sa90}) have 
become a fruitful source of interest. In particular, they have proved crucial in the theory
of Hodge invariants for the homotopy of complex algebraic varieties (see for example \cite{Mo}, \cite{ElZein}, \cite{Ha}, \cite{Na}).
A natural question arising is to ask for a homotopical structure in the category of such objects. 
Unfortunately, none of the contexts provided by the
derived categories of Verdier \cite{Ver} 
and the model categories of Quillen \cite{Q1}, considered nowadays as the standard basis of 
homological and homotopical algebra respectively, satisfy the needs to express the 
properties of diagram categories of complexes with filtrations.

In this paper we study the homotopy theory of mixed Hodge complexes
within the framework of Cartan-Eilenberg categories of \cite{GNPR}.
To achieve this, one must overcome two problems of distinct nature.

The first of these problems is to understand the different homotopical structures carried by filtered and bifiltered complexes.
Filtered derived categories were first studied by Illusie (see Chapter V of \cite{I})
following the classical theory for abelian categories.
An alternative approach in the context of exact categories was developed by Laumon \cite{Lau}.
In certain situations, the filtrations under study are not well defined, and become a proper invariant only in
higher stages of the associated spectral sequences. This is the case of the mixed Hodge theory of Deligne,
in which the weight filtration of a variety depends on the choice of a hyperresolution, and 
is only well defined at the second stage. This circumstance is somewhat hidden by the degeneration 
of the spectral sequences, but it already 
highlights the interest of studying more flexible structures.
In this paper we generalize
the results of Illusie by considering the class of weak equivalences
given by morphisms of filtered complexes inducing an isomorphism 
at a fixed stage of the associated spectral sequence (see also \cite{HT} and \cite{P}). 

The second problem is to obtain a theory of
rectification of morphisms of diagrams up to homotopy, allowing the construction of level-wise fibrant models for diagram categories.
This problem
is of great interest in the
field of abstract homotopical algebra, and has only been solved for some specific situations 
by means of Quillen-type theories (see for example \cite{Hir}, \cite{Br}, \cite{Ci}). 
In this paper we address the problem for diagrams of complexes over additive categories admitting a Cartan-Eilenberg structure and satisfying
certain compatibility conditions.

In a subsequent paper we will develop a more general theory
extending the main results of this paper to multiplicative 
mixed Hodge complexes.
In particular, we will show how Morgan's theory \cite{Mo} on the existence of mixed Hodge structures
in rational homotopy can be understood as a multiplicative version of Belinson's Theorem (see \cite{Be}, Theorem 3.4).
\\

We explain the main results of this paper.
The category  $\mathbf{F}\Aa$ of filtered objects (with finite filtrations) of an abelian category $\Aa$ is additive, but not abelian in general.
Consider the category $\mathbf{C}^{\#}(\mathbf{F}\Aa)$ of complexes over $\mathbf{F}\Aa$,
where ${\#}$ denotes the boundedness condition.
For $r\geq 0$, denote by
$\Ee_r$ the class of \textit{$E_r$-quasi-isomorphisms}: these are
morphisms of filtered complexes inducing a quasi-isomorphism at the $r$-stage of the associated spectral sequence.
The \textit{$r$-derived category} is defined by
$\mathbf{D}^\#_r(\mathbf{F}\Aa):=\mathbf{C}^{\#}(\mathbf{F}\Aa)[\Ee_r^{-1}].$
The case $r=0$ corresponds to the original filtered derived category, studied by Illusie in \cite{I}.
In order to deal with the weight filtration, Deligne \cite{DeHII} introduced the d\'{e}calage of a filtered complex, which shifts 
the associated spectral sequence of the original filtered complex by one stage. This defines a functor 
$\Dec:\mathbf{C}^{\#}(\mathbf{F}\Aa)\to\mathbf{C}^{\#}(\mathbf{F}\Aa)$
which is the identity on morphisms and satisfies $\Dec(\Ee_{r+1})\subset\Ee_r$.
We prove:
\newtheorem*{c21}{\normalfont\bfseries Theorem $\textbf{\ref{cxos_equiv_r}}$}
\begin{c21}
For all $r\geq 0$, Deligne's d\'{e}calage induces an equivalence of categories
$$\Dec:\mathbf{D}^{\#}_{r+1}(\mathbf{F}\Aa)\stackrel{\sim}{\lra} \mathbf{D}^{\#}_{r}(\mathbf{F}\Aa).$$
\end{c21}
The notion of homotopy between morphisms of complexes over an additive category is defined via a translation functor.
In the filtered setting, we find that an $r$-shift on
the filtration of the translation functor leads to different notions of \textit{$r$-homotopy}, suitable to the study of the $r$-derived category.
The associated class $\Ss_r$ of \textit{$r$-homotopy equivalences} satisfies $\Ss_r\subset\Ee_r$.

As in the classical case, we address the study of the $r$-derived category of filtered objects
$\mathbf{F}\Aa$ under the assumption that $\Aa$ has enough injectives.
Denote by
$\FCxinj{r}{\Aa}$
the full subcategory of those bounded below filtered complexes $(K,F)$ over injective objects of $\Aa$ whose differential
satisfies $dF^pK\subset F^{p+r}K$, for all $p\in\ZZ$. 
We prove:
\newtheorem*{c23}{\normalfont\bfseries Theorem $\textbf{\ref{r_filt_ab}}$}
\begin{c23}
Let $\Aa$ be an abelian category with enough injectives. For all $r\geq 0$ the
 triple $(\Cx{\mathbf{F}\Aa},\Ss_r,\Ee_r)$ is a (right) Cartan-Eilenberg category with fibrant models in
$\mathbf{C}_r^+(\mathbf{F}\Inj\Aa)$.
The inclusion induces an equivalence of categories
$$\mathbf{K}_r^+(\mathbf{F}\Inj\Aa)\stackrel{\sim}{\lra}\mathbf{D}_r^+(\mathbf{F}\Aa).$$
between the category of $r$-injective complexes modulo $r$-homotopy, and the localized category of
 filtered complexes with respect to $E_r$-quasi-isomorphisms.
\end{c23}
Denote by $\MHC$ the category of mixed Hodge complexes (see \cite{DeHIII}, Definition 8.1.5).
The spectral sequences associated with the weight and the Hodge filtrations of every mixed Hodge complex degenerate at the stages $E_2$ and $E_1$ respectively.
It proves to be more convenient to work with the category
$\AHC$ of absolute Hodge complexes as introduced by Beilinson \cite{Be}, in which
all spectral sequences
degenerate at the first stage. 
Deligne's d\'{e}calage with respect to the weight filtration induces a functor 
$\Dec_W:\MHC\to\AHC$, and
the cohomology of every absolute Hodge complex is a graded object in the category of mixed Hodge structures.

Our interest is to study the homotopy categories
$$\Ho(\MHC):=\MHC[\Qq^{-1}]\text{ and } \Ho(\AHC):=\AHC[\Qq^{-1}]$$
defined by inverting the class $\Qq$ of level-wise quasi-isomorphisms.
Denote by $\pi\left(\mathbf{G}^+(\MHS)^h\right)$ the category whose objects are non-negatively graded mixed Hodge structures and whose morphisms
are given by homotopy classes of level-wise morphisms compatible up to a filtered homotopy (ho-morphisms for short).
Denote by $\Hh$ the class of morphisms of absolute Hodge complexes that
are homotopy equivalences as ho-morphisms.
We prove:
\newtheorem*{c31}{\normalfont\bfseries Theorem $\textbf{\ref{AHC_sull}}$}
\begin{c31}
The triple $(\AHC,\Hh,\Qq)$ is a (right) Cartan-Eilenberg category, and $\mathbf{G}^+(\MHS)$ is a full subcategory of fibrant minimal models. 
The inclusion induces an equivalence of categories
$$\pi\left(\mathbf{G}^+(\MHS)^h\right)\stackrel{\sim}{\lra}\Ho(\AHC).$$
\end{c31}
Note that although every absolute Hodge complex is quasi-isomorphic to its cohomology
(which has trivial differentials), the full subcategory 
of fibrant minimal models has non-trivial homotopies. This reflects the fact
that mixed Hodge structures have non-trivial extensions.

We endow the category $\MHC$ with a Cartan-Eilenberg structure
(see Theorem $\ref{MHC_sull}$). The following result relates the two points
of view on mixed Hodge complexes which occur in the literature.
\newtheorem*{c32}{\normalfont\bfseries Theorem $\textbf{\ref{AHCequivMHC}}$}
\begin{c32}
Deligne's d\'{e}calage induces an equivalence of categories
$$\Dec_W:\Ho(\MHC)\stackrel{\sim}{\lra}\Ho(\AHC).$$
\end{c32}

We apply Theorem $\ref{AHC_sull}$ to provide 
an alternative proof of Beilinson's Theorem on absolute Hodge complexes.

\newtheorem*{c39}{\normalfont\bfseries Theorem $\textbf{\ref{beilinson}}$}
\begin{c39}[\cite{Be}, Theorem. 3.4]
The inclusion induces an equivalences of categories
$$\mathbf{D}^+\left(\MHS\right)\stackrel{\sim}{\lra}\Ho\left(\AHC\right).$$
\end{c39}

Lastly, as an application of the above results we compute the
extensions of mixed Hodge structures (see Theorem $\ref{extensions}$)
and describe morphisms in $\Ho(\MHC)$ in terms of morphisms and extensions of mixed Hodge structures
(see Corollary $\ref{morfismesAHC}$).

\section{D\'{e}calage and Filtered Derived Categories}
Deligne \cite{DeHII} introduced the shift and d\'{e}calage of filtered complexes and
proved that their associated spectral sequences are related by a shift of indexing.
We collect some main properties of shift and d\'{e}calage which are probably known to experts, but 
which do not seem to have appeared in the literature. We introduce the $r$-derived category of filtered complexes as the localization of
(bounded below) filtered complexes with respect to $E_r$-quasi-isomorphisms and, using Deligne's d\'{e}calage functor,
we generalize results of Illusie for $r=0$, to an arbitrary $r\geq 0$,
within the framework of Cartan-Eilenberg categories.

\subsection{Homotopy in additive categories}\label{Homotopiaditiva}
Let $\Aa$ be an additive category and denote by $\mathbf{C}^{\#}(\Aa)$ the category of cochain complexes
of $\Aa$,
where ${\#}$ denotes the boundedness condition ($+$ and $-$ for bounded below and above respectively,
$b$ for bounded and $\emptyset$ for unbounded).

Recall that the classical translation functor $T:\mathbf{C}^{\#}(\Aa)\to\mathbf{C}^{\#}(\Aa)$  is
defined on objects by $T(X)^n=X^{n+1}$ and $d_{T(X)}^n=-d_{X^{n+1}}$ and on morphisms by
$T(f)^n=f^{n+1}$.
We next define the notion of a translation functor in the category of complexes
which is induced by an additive automorphism of $\Aa$.

\begin{defi}\label{traslacio}
Let $\alpha:\Aa\to \Aa$ be an additive automorphism of $\Aa$ with a natural transformation $\eta:\alpha\to 1$.
The \textit{translation functor induced by $\alpha$}
is the automorphism
$T_\alpha:\mathbf{C}^{\#}(\Aa)\to\mathbf{C}^{\#}(\Aa)$
given by the composition $T_\alpha:=T\circ\alpha=\alpha\circ T$.
\end{defi}
Such translation functor will
prove to be useful
in the context of complexes over filtered abelian categories,
in which a shift by $r$ on the filtration of the classical translation leads
to the different notions of $r$-homotopy, as we shall see in the following section.

For the rest of this section we fix a 
translation functor $T_\alpha$ induced by an automorphism $\alpha$ of $\Aa$ with a natural transformation $\eta:\alpha\to 1$.

\begin{defi}\label{htp_cxos}
Let $f,g:X\to Y$ be morphisms of complexes. An \textit{$\alpha$-homotopy} from $f$ to $g$ is a degree preserving map
$h:T_\alpha(X)\to Y$ such that $dh+hd=(g-f)\circ \eta_X$. We denote $h:f\simr{$\alpha$} g$.
\end{defi}

The additive operation between maps makes
the homotopy relation into an equivalence relation compatible with the composition. 

Denote by $[K,L]_\alpha$ the set of morphisms of complexes from $X$ to $Y$ modulo $\alpha$-homotopy, and by
$\mathbf{K}^{\#}_\alpha(\Aa):=\mathbf{C}^{\#}(\Aa)/\simr{$\alpha$}$
the corresponding quotient category.
\begin{defi}
A morphism $f:X\to Y$ is said to be an 
\textit{$\alpha$-homotopy equivalence} if 
there exists a morphism $g:Y\to X$ together
with $\alpha$-homotopies $fg\simr{$\alpha$} 1_{Y}$ and $gf\simr{$\alpha$} 1_X$. Denote by $\Ss_\alpha$ the class of $\alpha$-homotopy equivalences.
\end{defi}

The following are standard constructions useful in the study of the homotopy theory of
complexes over $\Aa$ (see for example Section III.3.2 of \cite{GMa}).
We will later generalize these constructions in Section 2, for diagrams of complexes.

\begin{defi}\label{double_mapping_cylinder}
Let $f:X\to Y$ and $g:X\to Z$ be two morphisms of complexes. The \textit{$\alpha$-double mapping cylinder of $f$ and $g$} is the complex
$\Cc yl_\alpha(f,g)=T_\alpha(X)\oplus Y\oplus Z$
with differential
$$
D=\left(
\begin{matrix}
\text{-}d&0&0\\
\text{-}\eta_Y\circ\alpha(f)&d&0\\
\eta_Z\circ\alpha(g)&0&d
\end{matrix}
\right).
$$
\end{defi}
Denote by $i:Z\to \Cc yl_\alpha(f,g)$, $j:Y\to \Cc yl_\alpha(f,g)$ and $k:T_\alpha(X)\to \Cc yl_\alpha(f,g)$ the maps defined by
inclusion into the corresponding direct summands.
Then $i$ and $j$ are morphisms of complexes and $k$ is an $\alpha$-homotopy
from $jf$ to $ig$. With these notations:

\begin{lem}\label{caracteritza_double_cyl}
For any complex $W$, the map
$$\Hom(\Cc yl_\alpha(f,g),W)\to\left\{(h,u,v)\,;\, u\in \Hom(Y,W),\, v\in\Hom(Z,W),\, h:u f\simr{$\alpha$} v g\right\}$$
defined by $t\mapsto (t k,t j,t i)$ is a bijection.
\end{lem}
\begin{proof}
An inverse $(h,u,v)\mapsto t$ is given by $t(x,y,z)=h(x)+u(y)+v(z)$.
\end{proof}

\begin{defi}\label{mapping_cylinder}
Let $f:X\to Y$ be a morphism of complexes. 
\begin{enumerate}[(1)]
 \item The \textit{$\alpha$-mapping cylinder of $f$} is the complex
$\Cc yl(f):=\Cc yl(f,1_X)=T_\alpha(X)\oplus Y\oplus X$.
\end{enumerate}

There is a commutative diagram of morphisms of complexes
$$
\xymatrix{
\ar[dr]_fX\ar[r]^-{i}&\Cc yl_\alpha(f)\ar[d]^-{p}&Y\ar[l]_-{j}\\
&Y\ar@{=}[ur]
}
$$
where as before $i$ and $j$ denote the inclusions and $p(x,y,z)=y+f(z)$.
\begin{enumerate}[(2)]
 \item The \textit{$\alpha$-mapping cone of $f$} is the complex
$C_\alpha(f):=\Cc yl_\alpha(0,f)=T_\alpha(X)\oplus Y$.
\end{enumerate}

For every complex $W$, the map 
$$\Hom(C_\alpha(f),W)\to\left\{(h,v)\,;\, v\in \Hom(Y,W),\, h:v f\simr{$\alpha$} 0\right\}$$
defined by $t\mapsto (t k,tj)$ is a bijection.
\end{defi}

\begin{defi}\label{cylinder}
The \textit{$\alpha$-cylinder} of a complex $X$ is the complex
$\Cyl_\alpha(X):=\Cc yl_\alpha(1_X).$
\end{defi}
It is an easy consequence of Lemma $\ref{caracteritza_double_cyl}$ that an
$\alpha$-homotopy $h:T_\alpha(X)\to Y$ between morphisms $f,g:X\to Y$ is 
equivalent to a morphism of complexes $H:\Cyl_\alpha(X)\to Y$ satisfying $Hj=f$ and $Hi=g$.

An important property of the cylinder is the following.
\begin{prop}\label{htp_equiv_proj}For every complex $X$, the 
map $p:\Cyl_\alpha(X)\to X$ defined by $p(x,y,z)=y+z$ is an $\alpha$-homotopy equivalence, with homotopy inverse $j$.
\end{prop}
\begin{proof}
An $\alpha$-homotopy
$h:jp\simr{$\alpha$}1:T_\alpha(\Cyl_\alpha(X))\to \Cyl_\alpha(X)$ is
given by $h(x,y,z)=(z,0,0)$.
\end{proof}

\begin{cor}\label{quocient_abelianes}
The categories
$\mathbf{K}^{\#}_\alpha(\Aa)$ and $\mathbf{C}^{\#}(\Aa)[\Ss_\alpha^{-1}]$
are canonically isomorphic.
\end{cor}
\begin{proof}
It follows from Proposition $\ref{htp_equiv_proj}$ together with Proposition 1.3.3 of \cite{GNPR}.
\end{proof}

\subsection{Filtered Complexes}
Let $\mathbf{F}\Aa$ denote the additive category of filtered objects of an abelian category $\Aa$.
Throughout this paper we will consider filtered complexes $(K,F)\in\mathbf{C}^{\#}(\mathbf{F}\Aa)$
whose filtration is regular and exhaustive: for each $n\geq 0$ there exists $q\in\ZZ$
such that $F^qK^n=0$ and $K=\bigcup_p F^pK$. 
We will denote by $E_r(K)$ the spectral sequence associated with a filtered complex $(K,F)$, 
omitting the filtration $F$ whenever there is no danger of confusion.
For the rest of this section we fix an integer $r\geq 0$.

\begin{defi}
A morphism of filtered complexes $f:K\to L$ is called \textit{$E_r$-quasi-isomorphism}
if the induced morphism
$E_r(f):E_r(K)\to E_r(L)$ is a quasi-isomorphism of complexes.
\end{defi}

Denote by $\Ee_r$ the class of $E_r$-quasi-isomorphisms. The
\textit{$r$-derived category of filtered complexes} is the localized category
$$\mathbf{D}^{\#}_r(\mathbf{F}\Aa):=\mathbf{C}^{\#}(\mathbf{F}\Aa)[\Ee_{r}^{-1}].$$

For $r=0$ we recover the notions of 
filtered quasi-isomorphism and filtered derived category
studied by Illusie in \cite{I} 
(see also \cite{K1} and \cite{Pas} for an account in the frameworks
 of exact categories and Cartan-Eilenberg categories respectively). There is a chain of functors
$$\mathbf{D}^{\#}_0(\mathbf{F}\Aa)\to
\mathbf{D}^{\#}_1(\mathbf{F}\Aa)\to \cdots\to
\mathbf{D}^{\#}_r(\mathbf{F}\Aa)\to\cdots\to
\mathbf{D}^{\#}(\mathbf{F}\Aa),$$
where the rightmost category denotes the localization with respect to quasi-isomorphisms.
Each of these categories keeps less and less information 
of the original filtered homotopy type.

\begin{nada}We next introduce a notion of homotopy suitable to the study of the $r$-derived category.
Given $(A,F)\in\mathbf{F}\Aa$ define the filtered object $(A,F(r))$ by letting
$F(r)^pA:=F^{p+r}A$.
This defines an automorphism $\alpha_r$ of $\mathbf{F}\Aa$, and the identity
defines a natural transformation $\alpha_r\to 1$.

We will denote by $T_r:\mathbf{C}^{\#}(\mathbf{F}\Aa)\to \mathbf{C}^{\#}(\mathbf{F}\Aa)$
the translation functor induced by the automorphism $\alpha_r$. For every filtered complex $(K,F)$ we have
$F^pT_r(K)^n=F(r)^{p}K^{n+1}=F^{p+r}K^{n+1}$.

Given morphisms $f,g:K\to L$, an \textit{$r$-homotopy} from $f$ to $g$
is given by a degree preserving filtered map $h:T_r(K)\to L$ such that $dh+hd=g-f$.
The condition that $h$ is compatible with the filtrations is equivalent to
$h(F^pK^{n+1})\subset F^{p-r}L^{n}$ for all $n\geq 0$ and all $p\in\ZZ$.
Therefore our notion of $r$-homotopy coincides with the notion of homotopy of level $r$ of \cite{CaEil}, pag. 321.

Denote by $\Ss_r$ the associated class of \textit{$r$-homotopy equivalences}.
By Proposition 3.1 of  \cite{CaEil} we have $\Ss_r\subset\Ee_r$.
Hence the triple $(\FCx{\Aa},\Ss_r,\Ee_r)$ is a category with strong and weak equivalences.
\end{nada}

\subsection{Deligne's D\'{e}calage Functor}

\begin{defi}\label{shift_defi}
The \textit{shift} of a filtered complex $(K,F)$ is the filtered complex $(K,SF)$ defined by
$SF^pK^n=F^{p-n}K^n.$
This defines a functor 
$S:\mathbf{C}^{\#}(\mathbf{F}\Aa)\to \mathbf{C}^{\#}(\mathbf{F}\Aa)$
which is the identity on morphisms.
\end{defi}
The shift functor does not admit an inverse, since the differentials would not necessarily be compatible
with filtrations. However, it has both a right and
a left adjoint: these are the d\'{e}calage and its dual construction.
\begin{defi}[\cite{DeHII}]\label{decalage_defi}
The \textit{d\'{e}calage} of a filtered complex $(K,F)$ is the complex $(K,\Dec F)$ given by
$$\Dec F^pK^n=F^{p+n}K^n\cap d^{-1}(F^{p+n+1}K^{n+1}).$$
The \textit{dual d\'{e}calage} is the filtered complex $(K,\Dec^*F)$ given by
$$\Dec^*F^pK^n=d(F^{p+n-1}K^{n-1})+F^{p+n}K^n.$$
These define functors 
$\Dec,\Dec^*:\mathbf{C}^{\#}(\mathbf{F}\Aa)\lra \mathbf{C}^{\#}(\mathbf{F}\Aa)$
which are the identity on morphisms.
\end{defi}
\begin{example}
Let $G$ denote the trivial filtration $0=G^1K\subset G^0K=K$ of a complex $K$.
Then $SG=\sigma$ is the b\^{e}te filtration, while $\Dec G=\Dec^* G=\tau$ is the canonical filtration.
\end{example}
The next result is a matter of verification.
\begin{lem}\label{jota0}The following identities are satisfied:
\begin{enumerate}[(1)]
\item $\Dec\circ S=1$, and $(S \Dec F)^p=F^p\cap d^{-1}(F^{p+1})$,
\item $\Dec^*\circ S=1$, and $(S \Dec^*F)^p=F^p+d(F^{p-1})$.
\end{enumerate}
In particular, there are natural transformations
$S\circ \Dec\to 1$ and $1\to S\circ \Dec^*$.
\end{lem}

As a consequence of the above lemma we obtain:
\begin{prop}\label{adjunt_dec} The functor $S$ is left adjoint to $\Dec$ and right adjoint to $\Dec^*$.
In particular:
$$
\Hom(SK,L)=\Hom(K,\Dec L),\quad
\Hom(\Dec^*K, L)=\Hom(K,SL).
$$
\end{prop}

We next show that for a particular type of complexes, the d\'{e}calage and its dual construction coincide, and define an inverse functor to the shift.
\begin{nada}
Denote by $\mathbf{C}_r^{\#}(\mathbf{F}\Aa)$ the full subcategory of $\mathbf{C}^{\#}(\mathbf{F}\Aa)$ of those
filtered complexes $(K,F)$ satisfying $d(F^pK)\subset F^{p+r}K.$
In particular, the induced differential at the $s$-stage of their associated spectral sequence is trivial for all $s<r$.
\end{nada}
\begin{lem}\label{crinversos}
The functors $\Dec=\Dec^*:\mathbf{C}^{\#}_{r+1}(\mathbf{F}\Aa)\rightleftarrows \mathbf{C}^{\#}_r(\mathbf{F}\Aa):S$
are inverse to each other.
\end{lem}
\begin{proof}
If $K\in \mathbf{C}^{\#}_{r+1}(\mathbf{F}\Aa)$ then $d(F^pK)\subset F^{p+1}K$. Hence
$\Dec F^{p}K^n=\Dec^*F^{p}K^n=F^{p+n}K^n$. Therefore $S\circ \Dec (K)=S\circ \Dec^* (K)=K$.
A simple verification shows that $\Dec K\in \mathbf{C}^{\#}_{r}(\mathbf{F}\Aa)$.
Conversely, if $K\in \mathbf{C}^{\#}_{r}(\mathbf{F}\Aa)$ it is straightforward that
$SK\in \mathbf{C}^{\#}_{r+1}(\mathbf{F}\Aa)$ and $\Dec\circ S=1$.
\end{proof}

From the definition of the shift it follows that $E_{r+1}^{p+n,-p}(SK)\cong E_{r}^{p,n-p}(K)$ for all $r\geq 0$.
Therefore we have $S(\Ee_{r+1})\subset\Ee_r$.
Moreover, by Proposition 1.3.4 of \cite{DeHII}
the canonical maps
$$E_{r+1}^{p,n-p}(\Dec K)\lra E_{r+2}^{p+n,-p}(K)\lra E_{r+1}^{p,n-p}(\Dec^* K)$$
are isomorphisms for all $r\geq 0$.
We have the identities
$\Ee_{r+1}=\Dec^{-1}(\Ee_{r})=(\Dec^*)^{-1}(\Ee_{r}).$

\begin{teo}\label{cxos_equiv_r}
For all $r\geq 0$, Deligne's d\'{e}calage induces an equivalence of categories
$$\Dec:\mathbf{D}^{\#}_{r+1}(\mathbf{F}\Aa)\stackrel{\sim}{\lra} \mathbf{D}^{\#}_{r}(\mathbf{F}\Aa).$$
\end{teo}
\begin{proof}
Consider the composite functor $\Jj_r:=(S^r\circ \Dec^r):\mathbf{C}^{\#}(\mathbf{F}\Aa)\to \mathbf{C}^{\#}_r(\mathbf{F}\Aa)$ and denote
 $i_r:\mathbf{C}_r^{\#}(\mathbf{F}\Aa)\hookrightarrow \mathbf{C}^{\#}(\mathbf{F}\Aa)$ the inclusion.
We first show that $\Jj_r$ and $i_r$ induce inverse equivalences
$$\Jj_r: \mathbf{D}^{\#}_r(\mathbf{F}\Aa)\rightleftarrows\mathbf{C}^{\#}_r(\mathbf{F}\Aa)[\Ee_r^{-1}]:i_r.$$
Indeed, since $1=\Jj_r\circ i_r$ it suffices to show that the map $\eps_r:i_r\circ \Jj_r\to 1$ induced by the counit of
the adjunction $S\dashv \Dec$, is an $E_r$-quasi-isomorphism.
By (1) of Lemma $\ref{jota0}$ one has $\Dec^r\circ (i_r\circ \Jj_r)=\Dec^r$ and $\Dec^r(\eps_r)$ is the identity morphism.
Since $\Ee_{r+1}=\Dec^{-1}(\Ee_r)$, it follows that the map $\eps_r$ is an $E_r$-quasi-isomorphism. Hence the above equivalence follows.

Since $\Dec(\Ee_{r+1})\subset \Ee_r$ and $S(\Ee_r)\subset\Ee_{r+1}$, the inverse functors of Lemma $\ref{crinversos}$
induce an equivalence between localized categories
$\Dec:\mathbf{C}^{\#}_{r+1}(\mathbf{F}\Aa)[\Ee_{r+1}^{-1}]\stackrel{\sim}{\lra}\mathbf{C}^{\#}_{r}(\mathbf{F}\Aa)[\Ee_r^{-1}]$.
We have a diagram of functors
$$
\xymatrix{
\ar[d]^{\Jj_{r+1}}_\wr\mathbf{D}^{\#}_{r+1}(\mathbf{F}\Aa)\ar[r]^{\Dec}&\mathbf{D}^{\#}_{r}(\mathbf{F}\Aa)\ar[d]^{\Jj_r}_\wr\\
\mathbf{C}^{\#}_{r+1}(\mathbf{F}\Aa)[\Ee_{r+1}^{-1}]\ar[r]^{\Dec}_\sim&\mathbf{C}^{\#}_{r}(\mathbf{F}\Aa)[\Ee_r^{-1}].
}
$$
where the bottom and vertical arrows are equivalences. Since $\Dec\circ S=1$ this diagram commutes.
\end{proof}

The following result is a matter of verification and establishes
the behaviour of $r$-homotopies and $r$-homotopy equivalences by shift and d\'{e}calage.

\begin{lem}\label{commuten_translation}
The diagram of functors
$$
\xymatrix{
\mathbf{C}^{\#}(\mathbf{F}\Aa)\ar[d]^{T_{r+1}}\ar[r]^{\Dec}&
\mathbf{C}^{\#}(\mathbf{F}\Aa)\ar[d]^{T_{r}}\ar[r]^S&
\mathbf{C}^{\#}(\mathbf{F}\Aa)\ar[d]^{T_{r+1}}\ar[r]^{\Dec^*}&
\mathbf{C}^{\#}(\mathbf{F}\Aa)\ar[d]^{T_r}\\
\mathbf{C}^{\#}(\mathbf{F}\Aa)\ar[r]^{\Dec}&
\mathbf{C}^{\#}(\mathbf{F}\Aa)\ar[r]^S&
\mathbf{C}^{\#}(\mathbf{F}\Aa)\ar[r]^{\Dec^*}&
\mathbf{C}^{\#}(\mathbf{F}\Aa)\\
}
$$
commutes. In particular:
\begin{enumerate}[(i)]
\item The functor $\Dec$ sends $(r+1)$-homotopies from $f$ to $g$ to $r$-homotopies from $\Dec f$ to $\Dec g$.
\item The functor $S$ sends $r$-homotopies from $f$ to $g$ to $(r+1)$-homotopies from $Sf$ to $Sg$.
\end{enumerate}
\end{lem}
\begin{cor}\label{dechomotopies}
We have inclusions 
$\Dec(\Ss_{r+1})$, $\Dec^*(\Ss_{r+1})\subset\Ss_r$ and $S(\Ss_{r})\subset\Ss_{r+1}.$
\end{cor}

\subsection{$\mathbf{r}$-Injective Models}
Generalizing the notion of filtered injective complex of Illusie, 
we introduce $r$-injective complexes and show that these are fibrant objects in the sense of \cite{GNPR}, Definition 2.2.1,
with respect to the classes of $r$-homotopy equivalences and $E_r$-quasi-isomorphisms.
We then prove the existence of $r$-injective models for bounded below filtered complexes
(a similar result is proved by Paranjape in \cite{P}), giving rise to a Cartan-Eilenberg structure.
Our results are based on the original results of Illusie,
using induction over $r\geq 0$ via Deligne's d\'{e}calage functor.

\begin{defi}
A filtered complex $(K,F)$ is called \textit{$r$-injective} if for all $p\in\ZZ$:
\begin{enumerate}[(i)]
\item the differential satisfies $d(F^pK)\subset F^{p+r}K$, that is $K\in \mathbf{C}_r^{\#}(\mathbf{F}\Aa)$.
\item the graded object $Gr^p_FK\in\Cx{\Inj \Aa}$ is a complex of injective objects of $\Aa$.
\end{enumerate}
\end{defi}
For $r=0$ the first condition becomes trivial and
 we recover the original notion of filtered complex of injective type introduced by Illusie.

Denote by $\FCxr{r}{\Inj\Aa}$ the full subcategory of $\FCx{\Aa}$ of $r$-injective complexes.

\begin{lem}\label{inversinjectius} The functors
$\Dec=\Dec^*:\FCxr{r+1}{\Inj\Aa}\rightleftarrows \FCxr{r}{\Inj\Aa}:S$ are inverse to each other.
\end{lem}
\begin{proof}
It follows from Lemma $\ref{crinversos}$ together with the identity
 $Gr^p_{SF}K^n=Gr^{p-n}_FK^n$.
\end{proof}

\begin{prop}\label{r_injectius}
Let $I$ be an $r$-injective complex. Every $E_r$-quasi-isomorphism $w:K\to L$ induces
a bijection
$w^*:[L,I]_r\lra [K,I]_r$
between $r$-homotopy classes of morphisms.
\end{prop}
\begin{proof}
The case $r=0$ follows from Lemma V.1.4.3 of \cite{I}.
We proceed by induction. Assume that $I$ is an $(r+1)$-injective complex.
By Lemma $\ref{commuten_translation}$ we have a diagram
$$
\xymatrix{
\ar[d]^{w^*}[L,I]_{r+1}\ar[r]^-{\Dec^*}&[\Dec^*L,\Dec^* I]_{r}\ar[d]^{w^*}\\
[K,I]_{r+1}\ar[r]^-{\Dec^*}&[\Dec^*K,\Dec^* I]_{r}
}
$$
where $\Dec^*I$ is $r$-injective by Lemma $\ref{inversinjectius}$.
By induction hypothesis, the vertical arrow on the right is a bijection. Hence
to conclude the proof it suffices to show that the horizontal arrows are bijections.
By Lemma $\ref{inversinjectius}$ we have $S\circ\Dec^*I=I$. Together with the adjunction $\Dec^*\dashv S$
of Proposition $\ref{adjunt_dec}$ we obtain
$$\Hom(K,I)=\Hom(K,S\circ\Dec^*I)=\Hom(\Dec^*K,\Dec^*I).$$
This identity is valid, not only for morphisms of complexes, but for degree preserving filtered maps.
Together with Lemma $\ref{commuten_translation}$ this gives a bijection between the set of $(r+1)$-homotopies 
$h:T_{r+1}(K)\to I$ from $f$ to $g$ and the set of $r$-homotopies $h:T_{r}(\Dec^*K)\to \Dec^*I$ from $\Dec^*f$ to $\Dec^*g$.
Therefore the horizontal arrows are bijections.
\end{proof}

\begin{prop}\label{exist_rinj}
Let $\Aa$ be an abelian category with enough injectives.
For every filtered complex $K$ of $\Cx{\mathbf{F}\Aa}$ there is an $r$-injective complex $I$
and an $E_r$-quasi-isomorphism 
$\rho:K\to I$.
\end{prop}
\begin{proof}
The case $r=0$ follows from Lemma V.1.4.4 of \cite{I}. We proceed by induction.  
Assume that there exists an 
$E_{r}$-quasi-isomorphism $\rho:\Dec^*K\to I$,
where $I$ is $r$-injective.
The adjunction $\Dec^*\dashv S$ of Proposition $\ref{adjunt_dec}$ gives an $E_{r+1}$-quasi-isomorphism $\rho:K\to SI$.
By Lemma $\ref{inversinjectius}$, $SI$ is $(r+1)$-injective.
\end{proof}

\begin{teo}\label{r_filt_ab}Let $\Aa$ be an abelian category with enough injectives. For all $r\geq 0$ the
 triple $(\Cx{\mathbf{F}\Aa},\Ss_r,\Ee_r)$ is a (right) Cartan-Eilenberg category with fibrant models in
$\mathbf{C}_r^+(\mathbf{F}\Inj\Aa)$.
The inclusion induces an equivalence of categories
$$\mathbf{K}_r^+(\mathbf{F}\Inj\Aa)\stackrel{\sim}{\lra}\mathbf{D}_r^+(\mathbf{F}\Aa).$$
between the category of $r$-injective complexes modulo $r$-homotopy, and the localized category of
 filtered complexes with respect to $E_r$-quasi-isomorphisms.
\end{teo}
\begin{proof}
By Proposition $\ref{r_injectius}$ every $r$-injective complex is a fibrant object
with respect to the classes $\Ss_r\subset\Ee_r$.
By Proposition $\ref{exist_rinj}$ every filtered complex has an $r$-injective model.
Hence the triple $(\Cx{\mathbf{F}\Aa},\Ss_r,\Ee_r)$ is a Cartan-Eilenberg category with fibrant models in
$\mathbf{C}_r^+(\mathbf{F}\Inj\Aa)$.
The equivalence of categories follows from Theorem 2.3.4 of \cite{GNPR}.
\end{proof}

\subsection{Bifiltered Complexes}
We shall consider bifiltered complexes $(K,W,F)\in\FFCx{\Aa}$ with $W$ an increasing filtration and $F$ a decreasing filtration.
For the sake of simplicity, and given our interest in Hodge theory, we shall only
describe the homotopy theory of bifiltered complexes with respect to the class of $E_{r,0}$-quasi-isomorphisms.
Denote by $\Dec_W$ the functor defined by taking the d\'{e}calage with respect to the filtration $W$ and leaving $F$ intact.
The functor $\Dec_W^*$ is defined analogously via the dual d\'{e}calage.
\begin{defi}
Denote by $\Ee_{0,0}$ the class of morphisms $f:K\to L$ of $\FFCx{\Aa}$ inducing isomorphisms
$H(Gr_p^WGr^q_Ff):H(Gr_p^WGr^q_FK)\to H(Gr_p^WGr^q_FL)$ for all $p,q\in\ZZ$.
Inductively over $r\geq 0$, define a class $\Ee_{r,0}$ of weak equivalences by letting
$$\Ee_{r+1,0}:=\Dec_W^{-1}(\Ee_{r,0})=(\Dec_W^*)^{-1}(\Ee_{r,0}).$$ 
Morphisms of $\Ee_{r,0}$ are called \textit{$E_{r,0}$-quasi-isomorphisms} of bifiltered complexes.
\end{defi}
For every integer $r\geq 0$, denote by
$\mathbf{D}^+_{r,0}(\mathbf{F}^2\Aa):=\FFCx{\Aa}[\Ee_{r,0}^{-1}]$
the localized category of bifiltered complexes with respect to the class of $E_{r,0}$-quasi-isomorphisms.
There is an obvious notion of $(r,0)$-translation functor induced by the 
automorphism of $\mathbf{F}^2\Aa$ that sends every bifiltered object $(A,W,F)$ to the bifiltered object $(A,W(r),F)$.
This defines a notion of $(r,0)$-homotopy. The associated class $\Ss_{r,0}$ of $(r,0)$-homotopy
equivalences satisfies $\Ss_{r,0}\subset\Ee_{r,0}$. Hence the triple $(\Cx{\mathbf{F}^2\Aa},\Ss_{r,0},\Ee_{r,0})$
is a category with strong and weak equivalences.

\begin{defi}
A bifiltered complex $(I,W,F)$ is called \textit{$(r,0)$-injective} if for all $p,q\in\ZZ$,
\begin{enumerate}[(i)]
\item the differential satisfies $d(W_pF^qK)\subset W_{p-r}F^qK$, and
\item the complex $Gr_p^WGr^q_FK$ is an object of $\Cx{\Inj \Aa}$.
\end{enumerate}
\end{defi}

\begin{teo}\label{r_bifilt_ab}Let $\Aa$ be an abelian category with enough injectives. For all $r\geq 0$ the
 triple $(\Cx{\mathbf{F}^2\Aa},\Ss_{r,0},\Ee_{r,0})$ is a (right) Cartan-Eilenberg category with fibrant models in 
$\mathbf{C}_{r,0}^+(\mathbf{F}^2\Inj\Aa)$. The inclusion induces an equivalence of categories
$$\mathbf{K}_{r,0}^+(\mathbf{F}^2\Inj\Aa)\stackrel{\sim}{\lra}\mathbf{D}_{r,0}^+(\mathbf{F}^2\Aa)$$
between the category of $(r,0)$-injective complexes modulo $(r,0)$-homotopy, and the localized category of
bifiltered complexes with respect to $E_{r,0}$-quasi-isomorphisms.
\end{teo}
\begin{proof}
The proof is analogous to that of Theorem $\ref{r_filt_ab}$.
\end{proof}

\section{Homotopy Theory of Diagrams of Complexes}
In this section we study the homotopy theory of diagram categories
whose vertexes are categories of complexes over additive categories, within the framework of Cartan-Eilenberg categories.
The homotopical structure of the vertex categories is transferred to the diagram
category with level-wise weak equivalences and level-wise fibrant models.
The results will be applied to mixed Hodge complexes in the following section.

\subsection{Diagrams of complexes}
\begin{defi}\label{catdiagrames}
Let
$\Cc:I\to\mathsf{Cat}$ be a functor from a small category $I$, to the category of categories $\mathsf{Cat}$.
Denote $\Cc_i:=\Cc(i)\in \mathsf{Cat}$ for all $i\in I$ and  
 $u_*:=\Cc(u)\in \mathsf{Cat}(\Cc_i,\Cc_j)$ for all $u:i\to j$.
The \textit{category of diagrams $\GCc$ associated with the functor $\Cc$} is defined as follows:
\begin{enumerate}[$\bullet$]
\item An object $X$ of $\GCc$ is given by a family of objects $\{X_i\in \Cc_i\}$, for all $i\in I$,
together with a family of morphisms $\varphi_u:u_*(X_i)\to X_j$, called \textit{comparison morphisms}, for every map $u:i\to j$.
Such an object is denoted as
$X=\left(X_i\stackrel{\varphi_u}{\dashrightarrow}X_j\right).$
\item A morphism $f:X\to Y$ of $\GCc$ is given by a family of morphisms $\{f_i:X_i\to Y_i\}$ of $\Cc_i$
for all $i\in I$, such that for every map $u:i\to j$ of $I$, the diagram
$$\xymatrix{
u_*(X_i)\ar[d]_{u_*(f_i)}\ar[r]^{\varphi_u}&X_j\ar[d]^{f_j}\\
u_*(Y_i)\ar[r]^{\varphi_u}&Y_j&
}$$
commutes in $\Cc_j$. Such a morphism is denoted $f=(f_i):X\to Y$.
\end{enumerate}
We will often write $X_i$ for $u_*(X_i)$ and $f_i$ for $u_*(f_i)$,
whenever there is no danger of confusion.
\end{defi}

\begin{rmk}The category of diagrams $\GCc$ associated with $\Cc$ is the category of sections of the projection functor $\pi:\int_I\Cc\to I$,
where $\int_I\Cc$ is the \textit{Grothendieck construction} of $\Cc$ (see \cite{Th}). 
If $\Cc:I\to\mathsf{Cat}$ is the constant functor $i\mapsto\Cc$ then
$\GCc=\Cc^I$ is the diagram category of objects of $\Cc$ under $I$.
\end{rmk}

\begin{nada}\label{indexcat}
We will restrict our study to diagram categories indexed by a finite directed category $I$
whose degree function takes values in $\{0,1\}$. That is:
\begin{enumerate}
 \item [(D$_0$)] There exists a \textit{degree function} $|\cdot|:\text{Ob}(I)\lra \{0,1\}$ such that
$|i|<|j|$ for every non-identity morphism
$u:i\to j$ of $I$.
\end{enumerate}
A finite category $I$ satisfying (D$_0$) is a particular case of a Reedy category for which $I^+=I$.
The main examples of such categories are given by finite zig-zags
$$
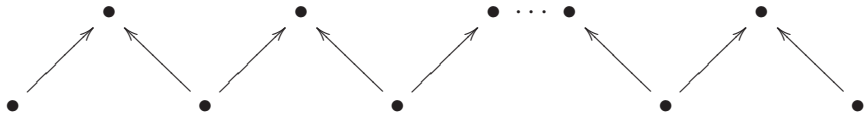
\xymatrix{
&\bullet&&\bullet&&\bullet\\
\bullet\ar[ur]&&\bullet\ar[ul]\ar[ur]&&\bullet\ar[ul]\ar[ur]
}\cdots\xymatrix{
\bullet&&\bullet\\
&\bullet\ar[ul]\ar[ur]&&\bullet\ar[ul]
}
$$
We will also assume that the functor $\Cc:I\to\mathsf{Cat}$ satisfies the following axioms:
\begin{enumerate}
\item[(D$_1$)] For all $i\in I$, $\Cc_i=\Cx{\Aa_i}$ is the category of bounded below complexes, where $\Aa_i$
is an additive category with an additive automorphism $\alpha_i:\Aa_i\to \Aa_i$ and a natural transformation $\eta_i:\alpha_i\to 1$.
In particular there is a class $\Ss_i$ of homotopy equivalences associated with the 
translation functor $T_{\alpha_i}$ induced by $\alpha_i$ (see Definition $\ref{traslacio}$).
\item[(D$_2$)] There is a class of weak equivalences $\Ww_i$ of $\Cc_i$ making the triple $(\Cc_i,\Ss_i,\Ww_i)$ 
into a category with strong and weak equivalences.
\item[(D$_3$)] For all $u:i\to j$, the functor $u_*$ is induced by an additive functor $A_i\to A_j$ compatible with $\alpha_i$. It
preserves fibrant objects, strong and weak equivalences.
\end{enumerate}
\end{nada}

With the above hypothesis, the category $\GCc$ is a category of complexes of diagrams of additive categories. 
The level-wise translation functors define a translation functor on $\GCc$.
This gives a notion of homotopy between morphisms of $\GCc$.
Denote by $\Ss$ the associated class of homotopy equivalences.
Note that if $f=(f_i)\in \Ss$, then 
$f_i\in \Ss_i$ for all $i\in I$, but the converse is not true in general.

Denote by $\Ww$ the class of level-wise weak equivalences of $\GCc$.
Since $\Ss_i\subset \Ww_i$ for all $i\in I$, it follows that $\Ss\subset \Ww$.
Therefore the triple $(\GCc,\Ss,\Ww)$ is a category with strong and weak equivalences.

\subsection{Morphisms up to homotopy}
We next introduce a new category $\GCc^h$ whose objects are those of $\GCc$ but whose morphisms
are defined by level-wise morphisms compatible up to fixed homotopies.

For the sake of simplicity, from now on we will omit the notations $T_i$, $\alpha_i$ and $\eta_i:\alpha_i\to 1$ regarding the translation functors
of $\Cc_i$, and use the standard notation $K[n]$ instead of $T_i^n(K)$, for $K\in \Cc_i$.
\begin{defi}\label{premorfisme}Let $X$ and $Y$ be two objects of $\GCc$ and let $n\in\ZZ$.
A \textit{pre-morphism of degree $n$ from $X$ to $Y$} is given by a pair of families
$f=(f_i,F_u)$, where 
\begin{enumerate}[(i)]
\item $f_i:X_i[-n]\to Y_i$ is a degree preserving map in $\Cc_i$, for all $i\in I$.
\item $F_u:X_i[-n+1]\to Y_j$ is a degree preserving map in $\Cc_j$, for $u:i\to j\in I$. 
\end{enumerate}
\end{defi}
\begin{nada}\label{cochaincomplex}
Given a pre-morphism $f=(f_i,F_u)$ of degree $n$ from $X$ to $Y$ we define its differential by
$$Df=\left(df_i-(-1)^nf_id, F_ud+(-1)^ndF_u+(-1)^n(f_j\varphi_u-\varphi_u f_i)\right).$$
\end{nada}
\begin{defi}A \textit{ho-morphism} $f:X\hto Y$ is a pre-morphism $f$ of degree $0$ from $X$ to $Y$ such that
 $Df=0$. Therefore it is given by a pair of families $f=(f_i,F_u)$ such that:
\begin{enumerate}[(i)]
\item $f_i:X_i\to Y_i$ is a morphism of complexes in $\Cc_i$ for all $i\in I$.
\item $F_u:X_i[1]\to Y_j$ satisfies $dF_u+F_ud=\varphi_uf_i-f_j\varphi_u$, for all $u:i\to j$, that is, $F_u$ is a homotopy from $f_j\varphi_u$ to $\varphi_uf_i$.

\end{enumerate}
\end{defi}
Given ho-morphisms $f:X\hto Y$ and $g:Y\hto Z$
let $gf:X\hto Z$ be the ho-morphism given by 
$gf=(g_if_i,G_uf_i+g_jF_u)$. This defines an associative
composition law between ho-morphisms. The identity ho-morphism of $X$ is $1=(1_{X_i},0)$.
A ho-morphism $f:X\hto Y$ is invertible if and only if
$f_i$ is invertible for all $i\in I$. Then the inverse of $f$ is given by
$f^{-1}=(f_i^{-1},-f_j^{-1}F_uf_i^{-1})$.

Denote by $\GCc^h$ the category whose objects are those of $\GCc$ and whose morphisms are ho-morphisms.
Every morphism $f=(f_i)$ can be made into a ho-morphism by letting $F_u=0$.
Hence $\GCc$ is a subcategory of $\GCc^h$.

\begin{defi}A ho-morphism $f=(f_i,F_u)$ is said to be a \textit{weak equivalence}
if $f_i$  is a weak equivalence in $\Cc_i$ for all $i\in I$.
\end{defi}
\begin{defi}\label{hohomotopies_cx} Let $f,g:X\hto Y$ be ho-morphisms.
A \textit{homotopy from $f$ to $g$} is a pre-morphism $h$ of degree $-1$ from $X$ to $Y$ such that
$Dh=g-f$. Hence $h=(h_i,H_u)$ is such that:
\begin{enumerate}[(i)]
 \item  $h_i:X_i[1]\to Y_i$ satisfies $dh_i+h_id=g_i-f_i$, that is, $h_i$ is a homotopy from $f_i$ to $g_i$.
\item  $H_u:X_i[2]\to Y_j$ satisfies $H_ud-dH_u=G_u-F_u+h_j\varphi_u-\varphi_uh_i$, for all $u:i\to j$.
\end{enumerate}
We denote such a homotopy by $h:f\simeq g$.
\end{defi}

 \begin{lem}
The homotopy relation between ho-morphisms is an equivalence relation, compatible with the composition.
\end{lem}
\begin{proof}Symmetry and reflexivity are trivial. To prove transitivity consider three ho-morphisms 
$f,f',f'':X\hto Y$ such that $h:f\simeq f'$, and $h':f'\simeq f''$.
A homotopy from $f$ to $f''$ is given by $$h''=h+h'=(h_i+h_i',H_u+H'_u).$$
This proves that $\simeq$ is an equivalence relation.
Let $h:f\simeq f'$ be a homotopy from $f$ to $f'$. Given a ho-morphism
$g:Y\hto Z$, a homotopy from $gf$ to $gf'$ is given by
$gh=(g_ih_i,G_uh_i+g_jH_u).$
Given a ho-morphism $g':W\hto X$, a homotopy from $fg'$ to $f'g'$ is given by
$hg'=(h_ig'_i,H_ug'_i+h_jG'_u).$
\end{proof}
We will denote by $[X,Y]^h$ the class of ho-morphisms from $X$ to $Y$ modulo homotopy and by
$\pi(\GCc^h):=\GCc^h/\simeq$
the corresponding quotient category $\pi(\GCc^h)(X,Y):=[X,Y]^h$.

The following generalize the constructions given in Section $\ref{Homotopiaditiva}$.
\begin{defi}
Let $f:X\hto Y$ and $g:X\hto Z$ be two ho-morphisms. The \textit{double mapping cylinder of $f$ and $g$} is the object of $\GCc$
defined by
$$\Cc yl(f,g)=\left(\Cc yl(f_i,g_i)\stackrel{\psi_u}{\dashrightarrow}\Cc yl(f_j,g_j)\right),$$
where $\Cc yl(f_i,g_i)$ is the double mapping cylinder of $f_i$ and $g_i$, and for all
$u:i\to j$ the comparison morphism $\psi_u:\Cc yl(f_i,g_i)\to \Cc yl(f_j,g_j)$
is given by
$$\psi_u=
\left(
\begin{matrix}
\varphi_u&0&0\\
\text{-}F_u&\varphi_u&0\\
G_u&0&\varphi_u
\end{matrix}
\right).
$$
\end{defi}
Let $i=(i_i,0):Z\to \Cc yl(f,g)$, $j=(j_i,0):Y\to \Cc yl(f,g)$ and $k=(k_i,0): X[1]\to \Cc yl(f,g)$ be defined
level-wise by the inclusions into the corresponding direct summands.
Then $i$ and $j$ are morphisms of diagrams and $k$ is a homotopy of ho-morphisms
from $jf$ to $ig$.
With these notations:

\begin{lem}\label{caracteritza_double_cyl_ho}
For any object $W$ of $\GCc$, the map
$$\GCc^h(\Cc yl(f,g),W)\to\left\{(h,u,v)\,;\, u\in \GCc^h(Y,W),\, v\in\GCc^h(Z,W),\, h:u f\simeq v g\right\}$$
defined by $t\mapsto (t k,t j,t i)$ is a bijection.
\end{lem}
\begin{proof}
Define an inverse $(h,u,v)\mapsto t=(t_i,T_u)$ by letting
$$t_i(x,y,z)=h_i(x)+u_i(y)+v_i(z)\text{ and }
T_u(x,y,z)=H_u(x)+U_u(y)+V_u(z).$$
\end{proof}
\begin{defi}
Let $f:X\hto Y$ be a ho-morphism.
\begin{enumerate}
 \item 
The \textit{mapping cylinder} of $f$ is the diagram
$$\Cc yl(f):=\Cc yl(f,1_X)=\left(\Cc yl(f_i)\stackrel{\psi_u}{\dashrightarrow}\Cc yl(f_j)\right).$$
 \item The \textit{mapping cone} of $f$ is the diagram
$$C(f):=\Cc yl(0,f)=\left(C(f_i)\stackrel{\psi_u}{\dashrightarrow}C(f_j)\right).$$
\end{enumerate}
\end{defi}

\begin{cor}\label{ho_cono}For any object $W$ of $\GCc$ the map
$$\GCc^h(C(f),W)\lra\{(h,v)\,;\,v\in\GCc^h(Y,Z),\, h:0\simeq vf\}$$
defined by $t\mapsto (t k,t j)$ is a bijection.
\end{cor}

\begin{defi}\label{cylinder_diag}
The \textit{cylinder} of a diagram $X$ is the diagram
$$\Cyl(X):=\Cc yl(1_X)=\left(\Cyl(X_i)\stackrel{\psi_u}{\dashrightarrow}\Cyl(X_j)\right).$$
\end{defi}
It is an easy consequence of Lemma $\ref{caracteritza_double_cyl_ho}$ that a
homotopy $h:f\simeq g$ between ho-morphisms $f,g:X\hto Y$ is 
equivalent to a ho-morphism $H:\Cyl (X)\hto L$ satisfying $Hj=f$ and $Hi=g$.

\subsection{Rectification of Ho-morphisms}
The notion of homotopy between ho-morphisms allows to define a new class of strong equivalences of $\GCc$ as follows.
\begin{defi}
A morphism  $f:X\to Y$ of $\GCc$ is said to be a 
\textit{ho-equivalence} if there exists a ho-morphism $g:Y\hto X$, together with homotopies
$gf\simeq 1_{X}$ and $fg\simeq 1_{Y}$. We say that the ho-morphism $g$ is a \textit{homotopy inverse} of $f$.
\end{defi}
Denote by $\Hh$ the class of ho-equivalences of $\GCc$. This class is closed by composition, 
and satisfies $\Ss\subset \Hh\subset \Ww$, where $\Ss$ and $\Ww$ denote the classes of homotopy and weak equivalences of $\GCc$.

Using the approach of Cartan-Eilenberg categories of \cite{GNPR} we study the localized category $\Ho(\GCc):=\GCc[\Ww^{-1}]$ by means of the localization
$\GCc[\Hh^{-1}]$. 

We will first describe $\GCc[\Hh^{-1}]$ in terms of homotopy classes of ho-morphisms.
Consider the solid diagram of functors
$$
\xymatrix{
\ar[d]_{\delta}\GCc\ar[r]^-i&\GCc^h\ar[r]^-{\pi}&\pi(\GCc^h)\\
\GCc[\Hh^{-1}]\ar@{.>}[urr]_{\Psi}
}
$$
Since every morphism of $\Hh$ is an isomorphism in $\pi(\GCc^h)$, there exists a unique dotted functor $\Psi$ making the diagram commute.
Our next objective is to show that $\Psi$ admits an inverse functor.

\begin{nada} 
Let $f:X\hto Y$ be a ho-morphism.
Define a ho-morphism $p:\Cc yl(f)\hto Y$ by the level-wise morphisms
$p_{i}(x,y,z)=y+f_i(z),$
together with the homotopies
$P_{u}(x,y,z)=F_u(z).$
\end{nada}

With these notations we have the following Brown-type factorization lemma for ho-morphisms.

\begin{prop}\label{factor_homor_cx}
Let $f:X\hto Y$ be a ho-morphism. Then the diagram
$$
\xymatrix{
X\ar@{~>}[dr]_{f}\ar[r]^-{i}&\Cc yl(f)\ar@{~>}[d]^-{p}&Y\ar[l]_-{j}\\
&Y\ar@{=}[ur]
}
$$
commutes in $\GCc^h$. In addition,
\begin{enumerate}[(1)]
\item  The maps $j$ and $p$ are weak equivalences.
\item The morphism $j$ is a ho-equivalence with homotopy inverse $p$.
\item  If $f$ is a weak equivalence, then $i$ is a weak equivalence.
\end{enumerate}
\end{prop}
\begin{proof}It is a matter of verification that the diagram commutes.
Since weak equivalences are defined level-wise, (1) and (3) are straightforward. We prove (2).
Let $h_i:\Cc yl(f_i)[1]\to \Cc yl(f_i)$ be defined by
$h_i(x,y,z)=(z,0,0).$
Then the pair $h=(h_i,H_u)$ with $H_u=0$,
is a homotopy from $jp$ to the identity. Indeed,
$(dh_i+h_id)=1-j_ip_i$ and $0=-j_jP_{u}+h_j\psi_u-\psi_uh_i$.
\end{proof}

We will also use the following result concerning ho-morphisms between mapping cylinders.
\begin{lem}\label{mapscilindres}
Given a commutative diagram in $\GCc^h$
$$\xymatrix{
X\ar@{~>}[d]_f\ar@{~>}[r]^a&Z\ar@{~>}[d]^g\\
Y\ar@{~>}[r]^{b}&W
}$$
there is an induced ho-morphism $(a,b)_*:\Cc yl(f)\hto \Cc yl(g)$
which is compatible with $i$, $j$ and $p$. In addition, if $a,b$ are morphisms of $\GCc$ then $(a,b)_*$ is also a morphism of $\GCc$.
\end{lem}
\begin{proof}
Since $g_ia_i=f_ib_i$, the assignation
$(x,y,z)\mapsto (a_i(x), b_i(y), a_i(z))$ gives a well defined morphism of complexes
$(a_i,b_i)_*:\Cc yl(f_i)\to \Cc yl(g_i)$. 
Since $G_ua_i+g_jA_u=b_jF_u+B_uf_i$, the assignation
$(x,y,z)\mapsto (-A_u(x),B_u(y),-A_u(z))$ defines a homotopy from $\psi_u(a_i,b_i)_*$ to $(a_j,b_j)_*\psi_u$.
We get a ho-morphism $(a,b)_*:\Cc yl(f)\hto \Cc yl(g)$ making the following diagram commute in $\GCc^h$.
$$
\xymatrix{
X\ar@{~>}[d]_{a}\ar[r]^-{i}&\ar@{~>}[d]_{(a,b)_*}\Cc yl(f)\ar@{~>}[r]^-{p}&Y\ar@{~>}[d]_{b}\ar[r]^-{j}&\Cc yl(f)\ar@{~>}[d]_{(a,b)_*}\\
Z\ar[r]^-{i}&\Cc yl(g)\ar@{~>}[r]^-{p}&W\ar[r]^-{j}&\Cc yl(g)\\
}
$$
If $a,b$ are morphisms of $\GCc$ then $A_u=0$ and $B_u=0$, hence $(\alpha,\beta)_*$ is also a morphism of $\GCc$.
\end{proof}

\begin{nada}\label{def_phi_cx}
Given two objects $X$ and $Y$ of $\GCc$, define a map
$$\Phi_{X,Y}:\GCc^h(X,Y)\lra \GCc[\Hh^{-1}](X,Y)$$
as follows.
Let $f:X\hto Y$ be a ho-morphism. By Proposition $\ref{factor_homor_cx}$ we have $f=p_fi_f$ in $\GCc^h$, where $i_f$ is a morphism of $\GCc$ and
$p_f$ is a homotopy inverse for the ho-equivalence $j_f$. Let
$$\Phi_{X,Y}(f):=j_f^{-1}i_f\in \GCc[\Hh^{-1}].$$
\end{nada}

\begin{teo}\label{phipsi}
The above map induces a well defined functor 
$\Phi:\pi(\GCc^h)\to \GCc[\Hh^{-1}]$
which is an inverse functor of $\Psi$.
\end{teo}
\begin{proof}
Since $\Ss\subset \Hh$, the functor $\delta:\GCc\to \GCc[\Hh^{-1}]$
factors through $\gamma:\GCc\to \GCc[\Ss^{-1}]$.
Hence to prove that two morphisms of $\GCc[\Hh^{-1}]$ coincide, it suffices to prove that they coincide in $\GCc[\Ss^{-1}]$.

We first show that the assignation $(X,Y)\mapsto \Phi_{X,Y}$ defines a functor $\Phi:\GCc^h\lra \GCc[\Hh^{-1}]$.
Given an object $X\in\GCc$, the diagram
$$
\xymatrix{
X\ar@{=}[rd]\ar[r]^-{i}&\ar[d]^-{p}Cyl(X)&X\ar[l]_-{j}\\
&X\ar@{=}[ur]
}
$$
commutes in $\GCc$. Hence $j^{-1}i=1_X$ in $\GCc[\Hh^{-1}]$, and $\Phi(1_X)=1_X$ for all $X\in \GCc$.

Let $f:X\hto Y$ and $g:Y\hto Z$ be two ho-morphisms. To show that $\Phi(g)\circ \Phi(f)=\Phi(g\circ f)$
consider the following diagram
$$\xymatrix{
\ar@{~>}[d]_a\Cc yl(f)&\ar@{~>}[d]_b\ar[l]_-{j_f}\ar[r]^-{i_g}Y&\ar@{~>}[d]_c\Cc yl(g)\\
\Cc yl(gf)\ar@{=}[r]&\Cc yl(gf)\ar@{=}[r]&\Cc yl(gf)
}$$
where $a:=(1_X,g)_*$, $b:=j_{gf}\circ g$ and $c:=j_{gf}\circ p_g$.
By Lemma $\ref{mapscilindres}$ the above diagram commutes in $\GCc^h$.
Taking factorizations we obtain a diagram of morphisms in $\GCc$
$$
\xymatrix{
&\Cc yl(gf)\ar[d]_{j_a}\ar@{=}[r]&\Cc yl(gf)\ar[d]_{j_b}\ar@{=}[r]&\Cc yl(gf)\ar[d]_{j_c}\\
X\ar[ur]^{i_{gf}}\ar[dr]_{i_{f}}\ar[r]&\Cc yl(a)&\Cc yl(b)\ar[l]_{(j_f,1)_*}\ar[r]^{(i_g,1)_*}&\Cc yl(c)&\ar[l]Y\ar[dl]^{j_g}\ar[ul]_{j_{gf}}\\
&\ar[u]^{i_a}\Cc yl(f)&\ar[l]_{j_f}\ar[r]^{i_g}Y\ar[u]^{i_b}&\Cc yl(g)\ar[u]^{i_c}
}
$$
which commutes in $\GCc[\Ss^{-1}]$. Therefore
$j_{gf}^{-1}i_{gf}=j_{g}^{-1}i_{g}j_{f}^{-1}i_{f}$ in $\GCc[\Hh^{-1}]$.
This proves that
$\Phi:\GCc^h\to \GCc[\Hh^{-1}]$ is a functor.

We next show that if $f\simeq g$ then $\Phi(f)=\Phi(g)$.
A homotopy from $f$ to $g$ defines a ho-morphism $h:\Cyl(X)\hto Y$ such that $hi_X=f$ and $hj_X=g$.
Consider the morphisms $f_*:=(i_X,1_Y)_*$ and $g_*:=(j_X,1_Y)_*$ defined by Lemma $\ref{mapscilindres}$.
Then the following diagram commutes in $\GCc[\Ss^{-1}]$.
$$
\xymatrix{
&\Cc yl(f)\ar[d]^{f_*}&\\
X\ar[dr]_{i_g}\ar[ur]^{i_f}\ar[r]^{i_h}&\Cc yl(h)&Y\ar[l]_{j_h}\ar[dl]^{j_g}\ar[ul]_{j_f}\\
&\Cc yl(g)\ar[u]_{g_*}&
}
$$
 Hence $\Phi(f)=\Phi(g)$.
This proves that $\Phi$ induces a well defined functor
$\Phi:\pi(\GCc^h)\to\GCc[\Hh^{-1}]$.

Lastly, we show that $\Phi$ is an inverse functor of $\Psi$.
Let $f:X\hto Y$ be a ho-morphism. Then
$\Psi(\Phi([f]))=\Psi(j_f^{-1}i_f)=[p_fi_f]=[f].$
To check the other composition it suffices to show that if $g:X\to Y$ is a ho-equivalence, then $\Phi(\Psi(g^{-1}))=g^{-1}$.
Let $h:Y\hto X$ be a homotopy inverse of $g$. By definition we have $\Psi(g^{-1})=[h]$. Then
$g\circ \Phi(\Psi(g^{-1}))=[g]\circ \Phi([h])=\Phi([g\circ h])=1$.
Therefore $\Phi(\Psi(g^{-1}))=g^{-1}$. This proves that $\Phi$ is an inverse functor of $\Psi$.
\end{proof}

\subsection{Fibrant Models of Diagrams}
\begin{nada}
Denote by $\GCc_{f}$ the full subcategory of $\GCc$ of those objects
$Q=(Q_i\stackrel{\varphi_u}{\dashrightarrow}Q_j)$
of $\GCc$ such that for all $i\in I$, $Q_i$ is fibrant in $(\Cc_i,\Ss_i,\Ww_i)$ in the sense of \cite{GNPR}, that is: 
every weak equivalence $f:X\to Y$ in $\Cc_i$
induces a bijection
$w^*:[Y,Q_i]\to [X,Q_i]$.
Condition (D$_3$) of $\ref{indexcat}$ ensures that for all $u:i\to j$, the object
$u_*(Q_i)$ is fibrant in $(\Cc_j,\Ss_j,\Ww_j)$.
\end{nada}
\begin{prop}\label{fibrants_ho}
Let $Q\in \GCc_{f}$ and let $w:X\hto Y$ be an ho-morphism. If $w$ is a weak equivalence then it induces a bijection
$w^*:[Y,Q]^h\lra[X,Q]^h$
between homotopy classes of ho-morphisms.
\end{prop}
\begin{proof}
We first prove surjectivity. Let $f:X\hto Q$ be a ho-morphism. Since $Q_i$ is fibrant in $\Cc_i$, there exists a morphism $g_i:Y_i\to Q_i$
together with a homotopy $h_i:g_iw_i\simeq f_i$, for all $i\in I$.
The chain of homotopies
$$g_j\varphi_uw_i\stackrel{-g_jW_u}{\simeq}g_jw_j\varphi_u\stackrel{h_j\varphi_u}
{\simeq}f_j\varphi_u\stackrel{F_u}{\simeq}\varphi_uf_i\stackrel{-\varphi_uh_i}{\simeq}\varphi_ug_iw_i$$
gives a homotopy
$G_u':=h_j\varphi_u-\varphi_uh_i+F_u-g_jW_u$
from $g_j\varphi_uw_i$ to $\varphi_ug_iw_i$.
By Lemma $\ref{caracteritza_double_cyl}$ the triple $(G_u', g_j\varphi_u,\varphi_ug_i)$ defines a morphism
$t_u:\Cc yl(w_i,w_i)\to Q_j$. We have a 
solid diagram
$$
\xymatrix{
\ar[d]_{w_i\oplus 1\oplus 1}\Cc yl(w_i,w_i)\ar[rr]^-{t_u}&&Q_j\\
\Cyl(Y_i)\ar@{.>}[urr]_{t_u'}&&.
}
$$
where the vertical arrow is a weak equivalence.
Since $u_*(Q_i)$ is fibrant in $\Cc_j$, 
there exists a dotted morphism $t_u'$ making the diagram commute up to homotopy.
By Lemma $\ref{caracteritza_double_cyl}$ $t_u'$ defines a triple $(G''_u,g'_u,g''_u)$ such that
$(H_u,h_u',h_u''):(G''_u,g'_u,g''_u)\circ (w_i\oplus 1\oplus 1)\simeq (G'_u,g_j\varphi_u,\varphi_ug_i)$.
Let $G_u:=-h'_u+G_u''+h_u''$. Then
$G_u:g_j\varphi_u\simeq \varphi_ug_i,$
and $H_u:G_uw_i\simeq G_u'$.
The pair $g=(g_i,G_u)$ is a ho-morphism, and $H=(h_i,H_u)$ is a homotopy from $gw$ to $f$.

To prove injectivity it suffices to show that if $f:Y\hto Q$ is such that $h:0\simeq fw$, then $0\simeq f$.
By Lemma $\ref{ho_cono}$ we have a solid diagram
$$
\xymatrix{
\ar@{~>}[d]_{w\oplus 1}C(w)\ar@{~>}[rr]^{(h,f)}&&Q\\
C(1_Y)\ar@{.>}[urr]_{(h',f')}&&.
}
$$
Since $w$ is a weak equivalence, the induced map
$(w\oplus 1)^*:[C(1_Y),Q]^h\lra [C(w),Q]^h$
is surjective. In particular there exists a ho-morphism $f':Y\hto Q$ such that $f'\simeq 0$, together with a homotopy
$f\simeq f'$. Using transitivity we have $f\simeq 0$.
Therefore the map $w^*$ is injective.
\end{proof}

\begin{cor}\label{esfibrant}
Let $Q\in \GCc_{f}$. Every weak equivalence $w:X\to Y$ in $\GCc$ induces a bijection
$$w^*:\GCc[\Hh^{-1}](Y,Q)\lra \GCc[\Hh^{-1}](X,Q).$$
\end{cor}
\begin{proof}
It follows from Proposition $\ref{fibrants_ho}$ and Theorem $\ref{phipsi}$.
\end{proof}

We next prove the existence of level-wise fibrant models in $\GCc$.
\begin{prop}
Let $\GCc$ be a category of diagrams satisfying the hypothesis (D$_0$)-(D$_3$) of $\ref{indexcat}$,
and assume that every object of $\Cc_i$ has a fibrant model in $(\Cc_i,\Ss_i,\Ww_i)$.
Then for every object $X$ of $\GCc$ there is an object $Q\in\GCc_f$, together with a
ho-morphism $X\hto Q$, which is a weak equivalence.
\end{prop}
\begin{proof}
Let $f_i:X_i\to Q_i$ be fibrant models in $(\Cc_i,\Ss_i,\Ww_i)$. By (D$_3$) of $\ref{indexcat}$,
$u_*(Q_i)$ is fibrant in $\Cc_j$. Therefore given the solid diagram
$$
\xymatrix{
X_i\ar[d]_{f_i}\ar[r]^{\varphi_u}&X_j\ar[d]^{f_j}\\
Q_i\ar@{.>}[r]^{\varphi_u'}&Q_j
}
$$
there exists a dotted arrow $\varphi_u'$ together with a homotopy $F_u:f_j\varphi_u\simeq \varphi_u'f_i$.
This defines an object
$Q$ of $\GCc_f$ and a ho-morphism
$f=(f_i,F_u):X\hto Q$ which is a weak equivalence.
\end{proof}

The following result is concerned with a full subcategory $\Dd\subset\GCc$ of a category of diagrams. 
This will be useful for the application to the category of mixed Hodge complexes, which is a subcategory of a convenient diagram category
satisfying a list of Hodge-theoretic axioms.

\begin{teo}\label{diag_cx_sull}
Let $\GCc$ be a category of diagrams satisfying the hypothesis {\normalfont(D$_0$)-(D$_3$)} of $\ref{indexcat}$.
Let $\Dd$ be a full subcategory of $\GCc$ such that:
\begin{enumerate}[(i)]
\item If $f:X\hto Y$ is a ho-morphism between objects of $\Dd$, the mapping cylinder $\Cc yl(f)$ is in $\Dd$.
\item There is a full subcategory $\Mm\subset\Dd\cap \GCc_{f}$ such that for every object $D$ of $\Dd$,
 there exists an object $Q\in\Mm$ together with a ho-morphism
 $D\hto Q$ which is a weak equivalence.
\end{enumerate}
Then the triple $(\Dd,\Hh,\Ww)$ is a right Cartan-Eilenberg category with fibrant models in $\Mm$. 
The inclusion induces an equivalence of categories
$$\pi(\Mm^h)\stackrel{\sim}{\lra} \Dd[\Ww^{-1}],$$
where $\pi(\Mm^h)$ is the category whose objects are those of $\Mm$ and whose morphisms
are homotopy classes of ho-morphisms.
\end{teo}
\begin{proof}
Consider the solid diagram of functors
$$
\xymatrix{
\ar[d]_{\gamma}\Dd\ar[r]^-\pi&\pi(\Dd^h)\\
\Dd[\Hh^{-1}]\ar@{.>}[ur]_{\Psi}&.
}
$$
Since every morphism of $\Hh$ is an isomorphism in $\pi(\Dd^h)$, by the universal
property of the localizing functor $\gamma$, there is a unique dotted functor $\Psi$ making the diagram commute.
By (i), the mapping cylinder $\Cc yl(f)$ of a ho-morphism $f:X\hto Y$ is in $\Dd$.
Hence the map $\Phi_{X,Y}:\Dd^h(X,Y)\to \Dd[\Hh^{-1}](X,Y)$ given by
$f\mapsto j_f^{-1}i_f$ is well defined.
By Theorem $\ref{phipsi}$ it induces a functor $\Phi:\pi(\Dd^h)\to \Dd[\Hh^{-1}]$ which is an inverse of $\Psi$.
In particular, by restriction to $\Mm$ we obtain an equivalence of categories $\Phi:\pi(\Mm^h)\stackrel{\sim}{\lra}\Mm[\Hh^{-1};\Dd]$,
where $\Mm[\Hh^{-1};\Dd]$ denotes the full subcategory of $\Dd[\Hh^{-1}]$ whose objects are in $\Mm$.

By (ii) and Proposition $\ref{esfibrant}$, for every object $D$ of $\Dd$ there exists a fibrant object $Q\in\Mm$ and a ho-morphism
$\rho:D\hto Q$ which is a weak equivalence. Then the morphism $\Phi_{D,Q}(\rho):D\to Q$ of $\Dd[\Hh^{-1}]$ is a fibrant model of $D$.
This proves that the triple $(\Dd,\Hh,\Ww)$ is a Cartan-Eilenberg category with fibrant models in $\Mm$.
By Theorem 2.3.4 of \cite{GNPR} the inclusion induces an equivalence of categories
$\Mm[\Hh^{-1};\Dd]\stackrel{\sim}{\lra}\Dd[\Ww^{-1}]$.
\end{proof}

\section{Mixed and Absolute Hodge Complexes}
\subsection{Diagrams of Filtered Complexes}
Let $I=\{0\to 1\leftarrow 2\to\cdots\leftarrow s\}$ be an index category of zig-zag type and
fixed length $s$, and let $\kk\subset\RR$ be a field.
We next define the category of diagrams of filtered complexes of type $I$ over $\kk$. This is a diagram category
whose vertexes are categories of filtered and bifiltered complexes defined over $\kk$ and $\CC$ respectively.
Additional assumptions on the filtrations will lead to the notions of mixed and absolute Hodge complexes.
\begin{defi}
Let $\mathbf{F}:I\to\mathsf{Cat}$ be the functor defined by
$$
\xymatrix{
0\ar@{|->}[d]\ar[r]^u&1\ar@{|->}[d]&\ar[l]\cdots\ar[r]&\ar@{|->}[d]s-1&\ar[l]_vs\ar@{|->}[d]\\
\FCx{\kk}\ar[r]^{u_*}&\FCx{\CC}&\ar[l]_-{Id}\cdots\ar[r]^-{Id}&\FCx{\CC}&\ar[l]_{v_*}\FFCx{\CC}
}
$$
where $u_*(K_\QQ,W)=(K_\QQ,W)\otimes\CC$ is given by extension of scalars
and $v_*(K_\CC,W,F)=(K_\CC,W)$ is the forgetful functor for the second filtration.
All intermediate functors are identities.
The \textit{category of diagrams of filtered complexes} is the
category of diagrams $\DFCx$ associated with $\mathbf{F}$. Objects of
$\DFCx$ are denoted by $K=((K_\kk,W)\stackrel{\varphi}{\dashleftarrow\dashrightarrow}(K_\CC,W,F))$.
\end{defi}

\begin{nada}For $r\in\{0,1\}$ consider the classes of weak equivalences of $\FCx{\kk}$ and $\FCx{\CC}$ given by $E_r$-quasi-isomorphisms,
and the notion of $r$-homotopy defined by the corresponding $r$-translation functors.
Likewise, in $\FFCx{\CC}$, consider the class of of weak equivalences given by $E_{r,0}$-quasi-isomorphisms,
and the notion of $(r,0)$-homotopy defined by the $(r,0)$-translation functor.
With these choices, the category of diagrams $\DFCx$ satisfies conditions (D$_0$)-(D$_3$) of $\ref{indexcat}$.
\end{nada}

Denote by $\Ee_{r,0}$ the class of level-wise weak equivalences of $\DFCx$.
The level-wise homotopies define the notions of \textit{ho$_r$-morphism} of diagrams and \textit{$(r,0)$-homotopy} between ho$_r$-morphisms.
Denote by $\Hh_{r,0}$ the corresponding class of ho-equivalences: 
these are morphisms of $\DFCx$ that are $(r,0)$-homotopy equivalences as ho$_r$-morphisms.
It satisfies $\Hh_{r,0}\subset \Ee_{r,0}$.

Denote by $\Qq$ the class of weak equivalences of $\DFCx$ given by level-wise quasi-isomorphisms compatible with filtrations.
Since filtrations are regular and exhaustive, 
we have $\Ee_{0,0}\subset \Ee_{1,0}\subset\Qq$.
Hence we have localization functors
$$\Ho_{0,0}\left(\DFCx\right)\to\Ho_{1,0}\left(\DFCx\right)\to\Ho_{}\left(\DFCx\right).$$

Deligne's d\'{e}calage with respect to the filtration $W$ defines a functor
$\Dec_W:\DFCx\to \DFCx$
which sends a diagram $K$ to the diagram
$\Dec_W K:=((K_\kk,\Dec W)\stackrel{\varphi}{\dashleftarrow\dashrightarrow}(K_\CC,\Dec W,F)).$
The functor $\Dec_W$ has a left adjoint $S_W:\DFCx\to \DFCx$ defined by shifting the filtration $W$. We have:
\begin{teo}\label{decw}
Deligne's d\'{e}calage induces an equivalence of categories
$$\Dec_W:\Ho_{1,0}\left(\DFCx\right)\lra \Ho_{0,0}\left(\DFCx\right).$$
\end{teo}
\begin{proof}
The proof follows from Theorems $\ref{cxos_equiv_r}$  and $\ref{r_bifilt_ab}$.
\end{proof}

 \subsection{Hodge Complexes}

\begin{defi}[\cite{DeHIII}, 8.1.5]\label{defMHC}
A \textit{mixed Hodge complex} is a diagram of filtered complexes
$$K=\left((K_\kk,W)\stackrel{\varphi}{\dashleftarrow\dashrightarrow}(K_\CC,W,F)\right),$$
satisfying the following conditions:
\begin{enumerate}
\item [($\text{MH}_0$)] The comparison map $\varphi$ is a string of $E_1^W$-quasi-isomorphisms. The cohomology $H^*(K_\kk)$ has finite type.
\item [($\text{MH}_1$)] For all $p\in\ZZ$, the differential of $Gr_p^WK_\CC$ is strictly compatible with the filtration $F$.
\item [($\text{MH}_2$)]  For all $n\geq 0$ and all $p\in\ZZ$, the filtration $F$ induced on $H^n(Gr^W_pK_{\CC})$ defines a pure Hodge structure of
weight $p+n$ on $H^n(Gr^W_pK_\kk)$.
\end{enumerate}
\end{defi}
Denote by $\MHC$ the category of mixed Hodge complexes of a fixed type $I$,
omitting the index category in the notation.
Note that axiom
($\mathrm{MH}_2$) implies that for all $n\geq 0$ the triple $(H^n(K_\kk),\Dec W,F)$ is a mixed Hodge structure.

An important result due to Deligne
is that given a mixed Hodge complex $K$,
the spectral sequences associated with $(K_\CC,F)$ and $(Gr_p^WK_\CC,F)$ degenerate at $E_1$, while the spectral sequences
associated with $(K_\kk,W)$ and $(Gr^q_FK_\CC,W)$ degenerate at $E_2$ 
(see Scholie 8.1.9 of \cite{DeHIII}).

For convenience, we consider a shifted version of mixed Hodge complex in which all spectral sequences degenerate at the first stage.
This corresponds to the notion of mixed Hodge complex given by Beilinson in \cite{Be}
(see also Section 2.3 of \cite{Le} and \cite{Sa00} where \textit{standard}
and \textit{absolute} weight filtrations are compared).

\begin{defi}
An \textit{absolute Hodge complex} is a diagram of filtered complexes
$$K=\left((K_\kk,W)\stackrel{\varphi}{\dashleftarrow\dashrightarrow}(K_\CC,W,F)\right),$$
satisfying the following conditions:
\begin{enumerate}
\item [($\text{AH}_0$)] The comparison map $\varphi$ is a string of $E_0^W$-quasi-isomorphisms. The cohomology $H^*(K_\kk)$ has finite type.
\item [($\text{AH}_1$)] The four spectral sequences associated with $K$ degenerate at $E_1$.
\item [($\text{AH}_2$)] For all $n\geq 0$ and all $p\in\ZZ$, the filtration $F$ induced on $H^n(Gr^W_pK_{\CC})$ defines a pure Hodge structure of
weight $p$ on $H^n(Gr^W_pK_\kk)$.
\end{enumerate}
\end{defi}
Denote by $\AHC$ the category of absolute Hodge complexes.
By Scholie 8.1.9 of \cite{DeHIII} Deligne's d\'{e}calage with respect to the weight filtration sends every mixed Hodge complex to an absolute Hodge complex.
Hence we have a functor $\Dec_W:\MHC\to\AHC$.
Note however that the shift $S_WK$ of an absolute Hodge complex $K$ is not in general a mixed Hodge complex.

Since the category of mixed Hodge structures is abelian (\cite{DeHII}, Theorem 2.3.5), 
every complex of mixed Hodge structures is an
absolute Hodge complex. We have full subcategories
$$\mathbf{G}^+(\MHS)\lra \Cx{\MHS}\lra\AHC.$$

Our objective is to study the homotopy categories
$$\Ho(\MHC):=\MHC[\Qq^{-1}]\text{ and } \Ho(\AHC):=\AHC[\Qq^{-1}]$$
defined by inverting level-wise quasi-isomorphisms.

The following result follows easily from Theorem 2.3.5 of \cite{DeHII}, stating
that morphisms of mixed Hodge structures are strictly compatible with filtrations.
\begin{lem}\label{quis_es_we}
We have $\Qq\cap \AHC=\Ee_{0,0}\cap \AHC$ and $\Qq\cap \MHC=\Ee_{1,0}\cap \MHC$.
In particular
$$\Ho_{0,0}(\AHC)=\Ho(\AHC)\text{ and }\Ho_{1,0}(\MHC)=\Ho(\MHC).$$
\end{lem}

\begin{lem}\label{tancades_quis}
Let $f:K\to L$ be an $E_{0,0}$- (resp. $E_{1,0}$-) quasi-isomorphism of $\DFCx$.
Then $K$ is an absolute (resp. mixed) Hodge complex if and only if $L$ is so.
\end{lem}
\begin{proof}
Assume that $f:K\to L$ is a an $E_{0,0}$-quasi-isomorphism and consider the diagram
$$
\xymatrix{
\ar[d]_{f_i}(K_i,W)\ar[r]^{\varphi_u^K}&(K_j,W)\ar[d]^{f_j}\\
(L_i,W)\ar[r]^{\varphi_u^L}&(L_j,W)&.
}
$$
Since the maps $f_i$ and $f_j$ are $E_0$-quasi-isomorphisms, by the two out of three property, it follows 
that $\varphi_u^K$ is an $E_0$-quasi-isomorphism if and only if $\varphi_u^L$ is so. Hence condition 
 ($\text{AH}_0$) is satisfied.
Condition ($\text{AH}_1$) is preserved by $E_{0,0}$-quasi-isomorphisms.
Condition  ($\text{AH}_2$) is a consequence of the following isomorphisms:
$$H^n(Gr_p^WGr^q_FK_\CC)\cong H^n(Gr_p^WGr^q_FL_\CC),\text{ and } 
H^n(Gr_p^WK_\kk)\cong H^n(Gr_p^WL_\kk).$$
 The proof for mixed Hodge complexes follows analogously.
\end{proof}

\subsection{Minimal Models}
The following technical lemma will be of use for the construction of minimal models.

\begin{lem}\label{seccions_compatibles}
Let $K$ be an absolute Hodge complex.
\begin{enumerate}[(1)]
 \item There are sections $\sigma^n_\kk:H^n(K_\kk)\to Z^n(K_\kk)$ and $\sigma^n_i:H^n(K_i)\to Z^n(K_i)$ of the projection, which are compatible with $W$.
\item There exists a section $\sigma^n_\CC:H^n(K_\CC)\to Z^n(K_\CC)$ of the projection, which is compatible with both filtrations $W$ and $F$.
\end{enumerate}
\end{lem}
\begin{proof}
The first assertion follows from the degeneration of the spectral sequence associated with $(K_\kk,W)$ at the first stage.
To prove the second assertion we use Deligne's splitting of mixed Hodge structures.
Since the cohomology $H^n(K_\kk)$ is a mixed Hodge structure,
by Lemma 1.2.11 of \cite{DeHII} there is a direct sum decomposition
$H^n(K_\CC)=\bigoplus I^{p,q}$
with $I^{p,q}\subset W_{p+q}F^pH^n(K_\CC)$ and such that
$$W_mH^n(K_\CC)=\bigoplus_{p+q\leq m+n} I^{p,q}\text{ and }F^lH^n(K_\CC)=\bigoplus_{p\geq l} I^{p,q}.$$
Therefore it suffices to define sections $\sigma^{p,q}:I^{p,q}\to Z^n(K_\CC)$.
By ($\text{AH}_1$)
the four spectral sequences associated with $K$ 
degenerate at $E_1$.
It follows that
the induced filtrations in cohomology are given by:
$$W_pF^qH^n(A_\CC)=\Img\{H^n(W_pF^qK_\CC)\to H^n(K_\CC)\}.$$
Since $I^{p,q}\subset  W_{p+q}F^qH^n(K_\CC)$ we have $\sigma^{p,q}(I^{p,q})\subset W_{p+q}F^pK_\CC$.
\end{proof}

\begin{teo}\label{minimals_AHC}
Let $K$ be an absolute Hodge complex. There is a ho$_0$-morphism 
$\rho:K\hto H(K)$ which is a quasi-isomorphism.
\end{teo}
\begin{proof}
By Lemma $\ref{seccions_compatibles}$ we can find sections
$\sigma_\kk:H^*(K_\kk)\to K_\kk$ and $\sigma_i:H^*(K_i)\to K_i$ compatible with the filtration $W$,
together with a section $\sigma_\CC:H^*(K_\CC)\to K_\CC$ compatible with $W$ and $F$. By definition, all maps are quasi-isomorphisms.
Let $\varphi_u:K_{i}\to K_{j}$ be a component of the quasi-equivalence $\varphi$ of $K$. The diagram
$$
\xymatrix{
H^*(K_i)\ar[d]_{\sigma_{i}}\ar[r]^{\varphi^*_u}&H^*(K_{j})\ar[d]_{\sigma_{j}}\\
K_{i}\ar[r]^{\varphi_u}&K_{j}
}
$$
is not necessarily commutative, but for any element $x\in H^*(K_{i})$, the difference
$(\sigma_{j}\varphi_u^*-\varphi_u\sigma_{j})(x)$
is a coborder. Since the differentials are strictly compatible with the weight filtration there exists a linear map
$\Sigma_u:H^*(K_{i})[1]\to K_{j}$
compatible with the weight filtration $W$ and such that
$\sigma_{j}\varphi_u^*-\varphi_u\sigma_{i}=d\Sigma_u.$
The morphisms
$\sigma_\kk$, $\sigma_i$ and $\sigma_\CC$ together with the homotopies $\Sigma_u$
define a ho$_0$-morphism $\sigma:H(K)\hto K$ which by construction is a quasi-isomorphism.
Since every object in $\AHC$ is fibrant, by Lemma $\ref{fibrants_ho}$ this lifts to a ho$_0$-morphism
 $\rho:K\hto H(K)$ which is a quasi-isomorphism.
\end{proof}

Denote by $\pi\left(\mathbf{G}^+(\MHS)^h\right)$ the category whose objects are non-negatively graded mixed Hodge structures and whose morphisms
are $(0,0)$-homotopy classes of ho$_0$-morphisms.
\begin{teo}\label{AHC_sull}
The triple $(\AHC,\Hh_{0,0},\Qq)$ is a Cartan-Eilenberg category and $\mathbf{G}^+(\MHS)$ is a full subcategory of fibrant minimal models. 
The inclusion induces an equivalence of categories
$$\pi\left(\mathbf{G}^+(\MHS)^h\right)\stackrel{\sim}{\lra}\Ho\left(\AHC\right).$$
\end{teo}
\begin{proof}
We show that the conditions of Theorem $\ref{diag_cx_sull}$ are satisfied, with
$\GCc=\Gamma\mathbf{F}$, $\Dd=\AHC$ and $\Mm=\mathbf{G}^+(\MHS)$.
By Lemma $\ref{quis_es_we}$, the class $\Qq$ of quasi-isomorphisms coincides with the class $\Ee_{0,0}$
of level-wise $E_0$- (resp. $E_{0,0}$-) quasi-isomorphisms.
Therefore conditions (D$_0$)-(D$_3$) of $\ref{indexcat}$ are trivially satisfied.
We prove condition (i) of Theorem $\ref{diag_cx_sull}$.
Let $f:K\hto L$ be a ho$_0$-morphism of $\AHC$. By Lemma $\ref{factor_homor_cx}$
the level-wise inclusion $j:L\to \Cc yl(f)$ is in $\Ee_{0,0}$.
It follows from Lemma $\ref{tancades_quis}$ that $\Cc yl(f)$ is an absolute Hodge complex.
Condition (ii) follows from
Theorem $\ref{minimals_AHC}$.
\end{proof}

Let $H$ be a graded object in the category of mixed Hodge structures.
The the shift $S_WH$ is a mixed Hodge complex with trivial differentials, since
$Gr^{SW}_pH_\kk^n=Gr^W_{p+n}H_\kk^n$. This gives a
functor $S_W:\mathbf{G}^+(\MHS)\to \MHC$.
Denote by 
$\pi (S_W(\mathbf{G}^+(\MHS))^h)$ the category whose objects are those of $S_W(\mathbf{G}^+(\MHS))$
and whose morphisms are $(1,0)$-homotopy classes of ho$_1$-morphisms.

\begin{teo}\label{MHC_sull}
The triple $(\MHC,\Hh,\Qq)$ is a Cartan-Eilenberg category, and $S_W(\mathbf{G}^+(\MHS))$ is a full subcategory of fibrant minimal models. 
The inclusion induces an equivalence of categories
$$\pi (S_W(\mathbf{G}^+(\MHS))^h)\stackrel{\sim}{\lra}\Ho\left(\MHC\right).$$
\end{teo}
\begin{proof}
It suffices to verify the hypothesis of Theorem 
$\ref{diag_cx_sull}$ for the subcategory $\MHC$ of $\DFCx$.
Conditions (D$_0$)-(D$_3$) of $\ref{indexcat}$  and condition (i) of Theorem $\ref{diag_cx_sull}$ follow analogously to 
Theorem $\ref{AHC_sull}$. We verify condition (ii).
Let $K$ be a mixed Hodge complex. Then $\Dec_WK$ is an absolute Hodge complex and by Theorem $\ref{minimals_AHC}$ there exists
a ho$_0$-morphism 
$\sigma:H(\Dec_W K)\hto \Dec_WK$.
The adjuncion $S_W\dashv\Dec_W$ defined at the level of diagrams of filtered complexes
together with Lemma $\ref{commuten_translation}.(ii)$
gives a ho$_1$-morphism 
$S_WH(\Dec_WK)\hto K$ which is a quasi-isomorphism.
\end{proof}

\begin{teo}\label{AHCequivMHC}
Deligne's d\'{e}calage induces an equivalence of categories
$$\Dec_W:\Ho(\MHC)\stackrel{\sim}{\lra}\Ho(\AHC).$$
\end{teo}
\begin{proof}
If suffices to note that when restricted to complexes with trivial differentials,
the functors $\Dec_W$ and $S_W$ are inverse to each other.
The result follows from Theorems $\ref{AHC_sull}$ and $\ref{MHC_sull}$.
\end{proof}

\subsection{Applications}
We provide
an alternative proof of Beilinson's Theorem on absolute Hodge complexes 
and study further properties
of morphisms of absolute Hodge complexes in the homotopy category.
\begin{teo}[\cite{Be}, Theorem. 3.4]\label{beilinson}
The inclusion induces an equivalence of categories
$$\mathbf{D}^+\left(\MHS\right)\stackrel{\sim}{\lra}\Ho\left(\AHC\right).$$
\end{teo}
\begin{proof}
It suffices to verify the hypothesis of Theorem 
$\ref{diag_cx_sull}$ for the subcategory $\mathbf{C}^+\left(\MHS\right)$ of $\DFCx$.
Conditions (D$_0$)-(D$_3$) of $\ref{indexcat}$  and condition (ii) of Theorem $\ref{diag_cx_sull}$ follow analogously to 
Theorem $\ref{AHC_sull}$. We verify condition (i), that is,
the mapping cylinder $\Cc yl(g)$
of a ho-morphism $g=(g_\kk,g_\CC,G):K\hto L$ of complexes of mixed Hodge structures is a complex of mixed Hodge structures.
The morphism $\psi:\Cc yl(g_\kk)\otimes\CC\to \Cc yl(g_\CC)$ defined by
$$
\psi=\left(\begin{matrix}
1&0&0\\
\text{-}G&1&0\\
0&0&1
\end{matrix}\right)
$$
is invertible. Define a filtration $F$ on $\Cc yl(f_\kk)\otimes\CC$ by letting
$F^p(\Cc yl(f_\kk)\otimes\CC):=\psi^{-1}(F^p\Cc yl(f_\CC)).$
Since the category of mixed Hodge structures is abelian, this endows $\Cc yl(f_\kk)$ with mixed Hodge structures.
\end{proof}

Every mixed Hodge structure can be identified with a complex of mixed Hodge structures concentrated in degree 0.
With this identification, and since the category $\MHS$ is abelian, given mixed Hodge structures $H$ and $H'$ over a field $\kk$, one can compute their 
extensions as
$$\Ext^{n}(H,H')={\mathbf{D}^+(\MHS)}(H,H'[n]).$$

Given filtered (resp. bifiltered) vector spaces $X$ and $Y$ over $\kk$,
denote by $\Hom^W(X,Y)$ (resp. $\Hom^W_F(X,Y)$ the set of morphisms
from $X$ to $Y$ that are compatible with the filtration $W$ (resp. the filtrations $W$ and $F$).

We next recover a result of Carlson \cite{Ca} regarding extensions of mixed Hodge structures, by studying
the morphisms in the homotopy category of absolute Hodge complexes (see also Section I.3 of \cite{PS} and Proposition 8.1 of \cite{Mo}).
\begin{teo}\label{extensions}
Let $H$ and $H'$ be mixed Hodge structures. Then
$$\Ext^1(H,H')={{\Hom^W(H_\CC,H'_\CC)}\over{\Hom^W(H_\kk,H'_\kk)+\Hom^W_F(H_\CC,H'_\CC)}},$$
and $\Ext^{n}(H,H')=0$ for all $n>1$.
\end{teo}
\begin{proof}
By Theorems $\ref{AHC_sull}$ and $\ref{beilinson}$ we have
$\Ext^{n}(H,H')=[H,H'[n]]^h$ for all $n\geq 0$.

A pre-morphism $g$ from $H$ to $H'[n]$ of degree $0$
is given by a triple $g=(g_\kk,g_\CC,G)$
where 
$$g_\kk\in\Hom^W(H_\kk,H_\kk'[n]),\,g_\CC\in\Hom^W_F(H_\CC,H_\CC'[n])\text{ and }
G\in\Hom^W(H_\CC,H_\CC'[n-1]).$$
Its differential is given by $Dg=(0,0,(-1)^m(g_\CC-g_\kk\otimes\CC)).$

If $n=1$ it follows that $g_\kk=0$, and $g_\CC=0$. In particular it satisfies $Dg=0$.
Such a cocycle $g=(0,0,G)$ is a coborder if and only if $G=h_\kk\otimes\CC-h_\CC$, where
$h_\kk\in\Hom^W(H_\kk,H_\kk')$ and $h_\CC\in\Hom^W_F(H_\CC,H_\CC').$
This proves the formula for $\Ext^1(H,H')$.

Lastly, if $n>1$ then $g_\kk=0$, $g_\CC=0$ and $G=0$. Hence $\Ext^n(H,H')=0$.
\end{proof}

Morphisms in the homotopy category of $\AHC$ are characterized as follows.

\begin{cor}\label{morfismesAHC}Let $K$ and $L$ be absolute Hodge complexes. Then
$$\Ho(\AHC)(K,L)=\bigoplus_n\left(\Hom_{\MHS}(H^nK,H^nL)\oplus\Ext^1_{\MHS}(H^nK,H^{n-1}L)\right).$$
\end{cor}
\begin{proof}
By Theorem $\ref{AHC_sull}$ there is a bijection
$\Ho(\AHC)(K,L)\cong [H(K),H(L)]^h.$
A ho-morphism $g:H(K)\hto H(L)$ is given by 
a morphism $g^*_\kk:H^*(K_\kk)\to H^*(L_\kk)$ compatible with $W$,
 a morphism $g_\CC^*:H^*(K_\CC)\to H^*(L_\CC)$ compatible with $W$ and $F$,
such that $g_\kk\otimes\CC\cong g_\CC$, together with
a morphism $G^*:H^*(K_\CC)[1]\to H^*(L_\CC)$ compatible with $W$.

Such a ho-morphism is a coboundary if $g=Dh$, for some pre-morphism $h$ of degree -1. 
This implies that $g_\kk=0$ and $g_\CC=0$, and that there exist
a morphism $h^*_\kk:H^*(K_\kk)[1]\to H^*(L_\kk)$ compatible with $W$,
and a morphism $h_\CC^*:H^*(K_\CC)[1]\to H^*(L_\CC)$ compatible with $W$ and $F$,
such that
$G\cong h_\kk\otimes\CC-h_\CC.$ The result now follows from Theorem $\ref{extensions}$.
\end{proof}

\subsection*{Acknowledgments}
We want to thank V. Navarro for his valuable comments and suggestions.

\linespread{1}
\bibliographystyle{amsalpha}
\bibliography{bibliografia}

\providecommand{\bysame}{\leavevmode\hbox to3em{\hrulefill}\thinspace}
\providecommand{\MR}{\relax\ifhmode\unskip\space\fi MR }
\providecommand{\MRhref}[2]{%
  \href{http://www.ams.org/mathscinet-getitem?mr=#1}{#2}
}
\providecommand{\href}[2]{#2}
\begin{thebibliography}{GNPR10}

\bibitem[Bei86]{Be}
A.~A. Beilinson, \emph{Notes on absolute {H}odge cohomology}, Applications of
  algebraic {$K$}-theory to algebraic geometry and number theory, vol.~55,
  1986, pp.~35--68.

\bibitem[Bro73]{Br}
K.~S. Brown, \emph{Abstract homotopy theory and generalized sheaf cohomology},
  Trans. Amer. Math. Soc. \textbf{186} (1973), 419--458.

\bibitem[Car80]{Ca}
J.~A. Carlson, \emph{Extensions of mixed {H}odge structures}, Journ\'ees de
  {G}\'eometrie {A}lg\'ebrique d'{A}ngers, {J}uillet 1979/{A}lgebraic
  {G}eometry, {A}ngers, 1979, 1980, pp.~107--127.

\bibitem[CE56]{CaEil}
H.~Cartan and S.~Eilenberg, \emph{Homological algebra}, Princeton University
  Press, Princeton, N. J., 1956.

\bibitem[Cis10]{Ci}
D-C. Cisinski, \emph{Cat\'egories d\'erivables}, Bull. Soc. Math. France
  \textbf{138} (2010), no.~3, 317--393.

\bibitem[Del71]{DeHII}
P.~Deligne, \emph{Th\'eorie de {H}odge. {II}}, Inst. Hautes \'Etudes Sci. Publ.
  Math. (1971), no.~40, 5--57.

\bibitem[Del74]{DeHIII}
\bysame, \emph{Th\'eorie de {H}odge. {III}}, Inst. Hautes \'Etudes Sci. Publ.
  Math. (1974), no.~44, 5--77.

\bibitem[EZ83]{ElZein}
F.~El~Zein, \emph{Mixed {H}odge structures}, Trans. Amer. Math. Soc.
  \textbf{275} (1983), no.~1, 71--106.

\bibitem[GM03]{GMa}
S.~Gelfand and Y.~Manin, \emph{Methods of homological algebra}, second ed.,
  Springer Monographs in Mathematics, Springer-Verlag, Berlin, 2003.

\bibitem[GNPR10]{GNPR}
F.~Guill{\'e}n, V.~Navarro, P.~Pascual, and A.~Roig, \emph{A
  {C}artan-{E}ilenberg approach to homotopical algebra}, J. Pure Appl. Algebra
  \textbf{214} (2010), no.~2, 140--164.

\bibitem[Hai87]{Ha}
R.~M. Hain, \emph{The de {R}ham homotopy theory of complex algebraic varieties.
  {II}}, $K$-Theory \textbf{1} (1987), no.~5, 481--497.

\bibitem[Hir03]{Hir}
P.~S. Hirschhorn, \emph{Model categories and their localizations}, Mathematical
  Surveys and Monographs, vol.~99, American Mathematical Society, Providence,
  RI, 2003. \MR{1944041 (2003j:18018)}

\bibitem[HT90]{HT}
S.~Halperin and D.~Tanr{\'e}, \emph{Homotopie filtr\'ee et fibr\'es {$C\sp
  \infty$}}, Illinois J. Math. \textbf{34} (1990), no.~2, 284--324.

\bibitem[Ill71]{I}
L.~Illusie, \emph{Complexe cotangent et d\'eformations. {I}}, Lecture Notes in
  Mathematics, Vol. 239, Springer-Verlag, Berlin, 1971.

\bibitem[Kel90]{K1}
B.~Keller, \emph{Chain complexes and stable categories}, Manuscripta Math.
  \textbf{67} (1990), no.~4, 379--417.

\bibitem[Lau83]{Lau}
G.~Laumon, \emph{Sur la cat\'egorie d\'eriv\'ee des {$\mathcal{D}$}-modules
  filtr\'es}, Algebraic geometry ({T}okyo/{K}yoto, 1982), Lecture Notes in
  Math., vol. 1016, Springer, Berlin, 1983, pp.~151--237.

\bibitem[Lev05]{Le}
M.~Levine, \emph{Mixed motives}, Handbook of {$K$}-theory. {V}ol. 1, 2,
  Springer, Berlin, 2005, pp.~429--521.

\bibitem[Mor78]{Mo}
J.~W. Morgan, \emph{The algebraic topology of smooth algebraic varieties},
  Inst. Hautes \'Etudes Sci. Publ. Math. (1978), no.~48, 137--204.

\bibitem[Nav87]{Na}
V.~Navarro, \emph{Sur la th\'eorie de {H}odge-{D}eligne}, Invent. Math.
  \textbf{90} (1987), no.~1, 11--76.

\bibitem[Par96]{P}
K.~H. Paranjape, \emph{Some spectral sequences for filtered complexes and
  applications}, J. Algebra \textbf{186} (1996), no.~3, 793--806.

\bibitem[Pas12]{Pas}
P.~Pascual, \emph{Some remarks on {C}artan-{E}ilenberg categories}, Collect.
  Math. \textbf{63} (2012), no.~2, 203--216.

\bibitem[PS08]{PS}
C.~Peters and J.~Steenbrink, \emph{Mixed {H}odge structures}, Ergebnisse der
  Mathematik und ihrer Grenzgebiete. 3. Folge. A Series of Modern Surveys in
  Mathematics, vol.~52, Springer-Verlag, Berlin, 2008.

\bibitem[Qui67]{Q1}
D.~Quillen, \emph{Homotopical algebra}, Lecture Notes in Mathematics, No. 43,
  Springer-Verlag, Berlin, 1967.

\bibitem[Sai90]{Sa90}
M.~Saito, \emph{Mixed {H}odge modules}, Publ. Res. Inst. Math. Sci. \textbf{26}
  (1990), no.~2, 221--333.

\bibitem[Sai00]{Sa00}
\bysame, \emph{Mixed {H}odge complexes on algebraic varieties}, Math. Ann.
  \textbf{316} (2000), no.~2, 283--331.

\bibitem[Tho79]{Th}
R.~W. Thomason, \emph{Homotopy colimits in the category of small categories},
  Math. Proc. Cambridge Philos. Soc. \textbf{85} (1979), no.~1, 91--109.

\bibitem[Ver96]{Ver}
J-L. Verdier, \emph{Des cat\'egories d\'eriv\'ees des cat\'egories
  ab\'eliennes}, Ast\'erisque (1996), no.~239, xii+253 pp. (1997).

\end{thebibliography}
\mbox{}\\
\linespread{1.2}

\end{document}